\documentclass[final,leqno]{siamart1116}

\usepackage{lineno,hyperref}
\usepackage{epsfig,amssymb,amsmath,version}
\usepackage{amssymb,version,graphicx,fancybox,mathrsfs,multirow}
\usepackage[notcite,notref]{showkeys}
\usepackage{url,hyperref}
\usepackage{subfigure,epstopdf,bm}
\usepackage{color}
\usepackage{accents}
\usepackage{appendix}
\usepackage{algorithm}
\usepackage{algorithmicx}
\usepackage{algpseudocode}
\usepackage{amsmath}
\usepackage{amsfonts}
\usepackage{amssymb}
\usepackage{graphicx}%

\theoremstyle{remark}
\newtheorem{rem}{Remark}[section]

\def \ri {{\rm i}}
\def \widebar {\accentset{{\cc@style\underline{\mskip10mu}}}}

\newcommand{\bs}[1]{\boldsymbol{#1}}

\title{Fast multipole method for 3-D Helmholtz equation in layered media\thanks{This work was supported by US Army Research Office (Grant No.W911NF-17-1-0368)
and US National Science Foundation (Grant No. DMS-1802143).}}

\author{Bo Wang\thanks{LCSM(MOE), School of Mathematics and Statistics, Hunan Normal University, Changsha, Hunan, 410081, P. R. China.
Department of Mathematics, Southern Methodist University, Dallas, TX 75275. This author acknowledges the financial support provided by NSFC (grant 11771137),
the Construct Program of the Key Discipline in Hunan Province and a Scientific Research Fund of Hunan Provincial Education Department (No. 16B154).}\and Wenzhong Zhang\thanks{Department of Mathematics, Southern Methodist University, Dallas, TX 75275.}
        \and Wei Cai\thanks{Corresponding author, Department of Mathematics, Southern Methodist University, Dallas, TX 75275({\tt cai@smu.edu}). {\it Submitted to SIAM J. Scientific Computing, March 1, 2019 and accepted for publication, August 26, 2019.}}}

\begin{document}
\maketitle

\begin{abstract}
In this paper, a fast multipole method (FMM) is proposed to compute long-range
interactions of wave sources embedded in 3-D layered media. The layered media Green's function for the Helmholtz equation, which satisfies the transmission conditions at material interfaces, is decomposed into a free space component and four types of reaction field components arising from wave reflections and transmissions through the layered media. The proposed algorithm is a combination of the classic FMM for the free space component and FMMs  specifically designed for the four types reaction components, made possible by new multipole expansions (MEs) and local expansions (LEs) as well as the multipole-to-local translation (M2L) operators for the reaction field components. { Moreover, equivalent polarization source can be defined for each reaction component based on the convergence analysis of its ME. The FMMs for the reaction components, implemented with the target particles and equivalent polarization sources, are found to be much more efficient than the classic FMM for the free space component due to the fact that the equivalent polarization sources and the target particles are always separated by a material interface.} As a result, the FMM algorithm developed for layered media has a similar computational cost as that for the free space. Numerical results validate the fast convergence of the MEs and the $O(N)$ complexity of the FMM for interactions of low-frequency wave sources in 3-D layered media.
\end{abstract}

\begin{keywords}
fast multipole method, layered media, multipole expansions, local expansions, Helmholtz equation, equivalent polarization sources
\end{keywords}

\begin{AMS}
15A15, 15A09, 15A23
\end{AMS}

\pagestyle{myheadings}
\thispagestyle{plain}
\markboth{Bo Wang, Wenzhong Zhang, and Wei Cai}{Fast multipole method for 3-D Helmholtz equation in layered media}

\section{Introduction}
The fast multipole method (FMM) has been a revolutionary development in modern
computational algorithms for treating many-body interactions. In fact, it was
considered one of the top 10 algorithms in the 20th century \cite{cipra2000best}. The FMM can reduce the $O(N^{2})$ cost of computing long-range interactions
(Columbic electrostatics, wave scattering) among $N$ particles (or sources) to $O(N)$ or
$O(N\log N)$. Such a capability of the FMM has had a tremendous impact on modern computational biology,
astronomy, and computational acoustics and electromagnetics, among many other
applications in sciences and engineering. The original FMMs developed by
Greengard and Rokhlin \cite{greengard1987fast,greengard1997new} for particles in the free space are based on multipole
expansions (MEs) for a Green's function\ $G(\bs r,\bs r^{\prime}%
)$=$G(\bs r-\bs r^{\prime})$ to achieve a low-rank representation
for the {long-range interactions between sources}. The MEs for the wave interactions were made
possible by the Graf's addition theorem for Bessel functions. As a simple way
to view the ME, we split the argument of the Green's function,
the zeroth order Bessel function for wave interactions, as $G(\bs r%
-\bs r^{\prime})=G((\bs r-\bs r_{c})+(\bs r_{c}%
-\bs r^{\prime}))$ where $\bs r_{c}$ is selected as the center of
many sources $\bs r^{\prime}$ while $\bs r$ is any field
location far away (well-separated) from all sources. The Graf's addition theorem gives an expansion
of $G$ in terms of separable terms involving ($\bs r-\bs r_{c}$) and
($\bs r_{c}-\bs r^{\prime}$), respectively. Terms involving
($\bs r-\bs r_{c}$) will appear as the higher order multipoles, i.e.,
higher order Bessel functions, while terms with ($\bs r_{c}-\bs r%
^{\prime}$) for each source $\bs r^{\prime}$ contributes  to the ME
expansion coefficients. For a far field location ${\bs r,}$ such an expansion
exhibits exponential convergence, thus only a small number of $p$ terms,
therefore $p$ ME-coefficients, are needed, resulting in a $p$-term
low-rank approximation for the far field of many sources.

{As the original FMM was developed based on far field approximations by MEs obtained by the Graf's addition theorem applied to free space Green's functions, we can understand the difficulties of
extending this approach to sources embedded in layered media, which are
ubiquitous in computer engineering, geophysical, and medical image applications.}
For those applications, Green's functions for layered media (layered Green's functions) are preferred to describe the interactions to avoid introducing artificial unknowns
on the infinite material interfaces. Unfortunately, in those cases, { no theory like Graf's addition theorem is available for layered Green's functions.  For this reason, the ME-based FMM
of Greengard and Rokhlin
has not been extended to layered Green's functions due to the lack of corresponding multipole expansions of the far field of wave sources in general multi-layered media.
In our recent work \cite{cho2018heterogeneous}, MEs for the
case of an impedance 2-D half space were obtained by using an analytical image representation of the layered Green's function, which however are not available for general multi-layered media. Then, in our more recent work \cite{zhang2018exponential},  MEs and multipole to local (M2L) translation operators were developed for 2-D Helmholtz equations in general layered media.
In a different approach, an inhomogeneous plane wave fast method was developed \cite{huchew2000} by approximating the Green's function
with plane waves sampled along a steepest descent path in the complex wave number space.}
So far, to handle
the wave interaction of sources embedded in general layered media using layered Green's
functions, some other fast methods have been proposed such as FMM using Taylor expansion-based low-rank representation of Green's function
 \cite{wang2018}\cite{tausch2003},
and cylindrical wave decomposition of the Green's function together with 2-D FMMs for cylindrical waves \cite{cho2012parallel}.
On the other hand, kernel independent compression techniques \cite{ying2004} \cite{krasny2019} could also be considered for
the layered Green's functions.

In this paper, we will develop MEs and local expansions (LEs) for
general layered Green's function for { 3-D Helmholtz equations as well as relevant M2L operators (thus extending our previous results for 2-D Helmholtz equations \cite{zhang2018exponential})},  providing the key ingredients in the hierarchical design of the 3-D FMM. { The layered Green's function is decomposed
into a free space part and four reaction fields arising from wave reflections
and transmissions through the layered media}. Our approach relies on two technical identities, the first one
expresses the layered Green's functions in terms of Sommerfeld
integral involving plane waves, and the second one is the Funk-Hecke identity, which expresses plane
waves in terms of cylindrical waves (Bessel functions). With the separable
property of plane waves as well as the Funk-Hecke identity, we are able to
derive the MEs and LEs for the layered Green's functions in 3-D as well as the M2L
operators. As shown in \cite{zhang2018exponential}, the convergence of the far
field for the reaction field component of the layered Green's function in fact depends on the distance between the target and the location
of some kind of polarization source, { which can be defined for each reaction component (see (\ref{eqpolarizedsource}) and Fig. \ref{sourceimages})}.
Therefore, in the implementation of the FMM for the reaction field components, the original
and the polarization sources will be combined and embedded into a rectangular box, upon which the oct-tree structure will be built and the FMM can be implemented. As a result, the FMM
for the free space can be extended straightforwardly to layered Green's functions. Numerical results will show the fast convergence of the MEs and LEs for the
layered Green's function as well as the $O(N)$ efficiency for 3-D low-frequency acoustic
wave interactions.

The rest of the paper is organized as follows. In section 2, we will
re-derive the ME, LE and relevant M2L operators for the free space Green's function using the new approach discussed above. The same technique will be then applied
to layered Green's functions. Section 3 gives the derivation of ME, LE and M2L operators for layered Green's functions, given as Sommerfeld integrations of plane waves.
A FMM based on the new ME, LE and M2L operator for reaction components is then presented. Section 4 gives numerical results for sources in three layers media.
Various efficiency comparisons are given to show the performance of the proposed FMM. Finally, conclusion and discussions are given in Section 5.

\section{A new derivation for $h$- and $j$-expansions of the free space Green's function of 3-D Helmholtz equation and translations}
In this section, we first review the $h$- and $j$-expansions, namely the multipole and local expansions, for the free space Green's function of the Helmholtz equation and the translation from a $h$-expansion to a $j$-expansion.  They are the key formulas in the free space FMM and can be derived by using the well-known addition theorems for regular and singular wave functions. Then, we will introduce a new derivation for the $h$- and $j$-expansions by using an integral representation of Hankel functions. This new technique will be applied to derive multipole and local expansions for reaction components of layered Green's functions in the next section.
\subsection{The $h$- and $j$-expansions of free space Green's function}
Let us first recall the addition theorems for regular and singular wave functions.
Suppose $\bs r_j=(r_j, \theta_j,\varphi_j)$ is the position vector of a point $P$ in spherical coordinates with respect to a given center $O_j$ for $j=1, 2$, and $\bs b=(b, \alpha,\beta)$ is the position vector of $O_1$ with respect to $O_2$, such that $\bs r_2=\bs r_1+\bs b$.
Meanwhile, the spherical harmonics are defined as
\begin{equation}
Y_n^m(\theta,\varphi)=(-1)^m\sqrt{\frac{2n+1}{4\pi}\frac{(n-m)!}{(n+m)!}}P_n^m(\cos\theta)e^{\ri m\varphi}=\widehat P_n^m(\cos\theta)e^{\ri m\varphi},
\end{equation}
where $0 \le |m| \le n$ and $P_n^m(\cos\theta)$ is the associated Legendre function and $\widehat P_n^m(\cos\theta)$ is its normalized version. We will use the following addition theorems \cite{martin2006multiple}.
\begin{theorem}\label{zeroorderaddthm}
	Let $z_{\nu}(\omega)$ be any spherical Bessel function of order $\nu$, that is
	$$z_{\nu}(\omega)=j_{\nu}(\omega), y_{\nu}(\omega), h_{\nu}^{(1)}(\omega), \hbox{or}\; h_{\nu}^{(2)}(\omega).$$
	Let $\bs r_2=\bs r_1+\bs b$. Then
	\begin{equation}\label{sphbessel00}
	z_0(kr_2)=
	 \begin{cases}
	 \displaystyle 4\pi\sum\limits_{n=0}^{\infty}\sum\limits_{m=-n}^n(-1)^nz_n(kb)\overline{Y_n^m(\alpha,\beta)}j_n(kr_1)Y_n^m(\theta_1,\varphi_1),\quad{\rm if}\; r_1<b, \\
	 \displaystyle 4\pi\sum\limits_{n=0}^{\infty}\sum\limits_{m=-n}^n(-1)^nj_n(kb)\overline{Y_n^m(\alpha,\beta)}z_n(kr_1)Y_n^m(\theta_1,\varphi_1),\quad{\rm if}\; r_1>b.
	 \end{cases}
	\end{equation}
\end{theorem}
The special case of the above theorem when $z_0=j_0$ was proved by Clebsch in 1863, see \cite[p.363]{watson}. According to Watson, another special case with $z_0=y_0$ was due to Gegenbauer.
\begin{theorem}\label{additionthmbesselj}
	Let $\bs r_2=\bs r_1+\bs b$. Then
	\begin{equation}
	j_n(kr_2)Y_n^m(\theta_2,\varphi_2)=\sum\limits_{\nu=0}^{\infty}\sum\limits_{\mu=-\nu}^{\nu}\widehat{S}_{n\nu}^{m\mu}(\bs b)j_{\nu}(kr_1)Y_{\nu}^{\mu}(\theta_1,\varphi_1),
	\end{equation}
	where
	\begin{equation}
	\label{spmhat3d}
	\widehat{S}_{n\nu}^{m\mu}(\bs b)=4\pi(-1)^m\ri^{\nu-n}\sum\limits_{q=0}^{\infty}\ri^qj_q(kb)\overline{Y_q^{\mu-m}(\alpha,\beta)}\mathcal{G}(n,m;\nu,-\mu;q),
	\end{equation}
	with $\mathcal{G}(n,m;\nu,-\mu;q)$ being the Gaunt coefficient.
\end{theorem}

\begin{theorem}
	\label{thmhlynm}
	Let $\bs r_2=\bs r_1+\bs b$. Then
	\begin{equation}
	h^{(1)}_n(kr_2)Y_n^m(\theta_2,\varphi_2)=\sum\limits_{\nu=0}^{\infty}\sum\limits_{\mu=-\nu}^{\nu}{S}_{n\nu}^{m\mu}(\bs b)j_{\nu}(kr_1)Y_{\nu}^{\mu}(\theta_1,\varphi_1),
	\end{equation}
	for $r_1<b$, and
	\begin{equation}
	h^{(1)}_n(kr_2)Y_n^m(\theta_2,\varphi_2)=\sum\limits_{\nu=0}^{\infty}\sum\limits_{\mu=-\nu}^{\nu}\widehat{S}_{n\nu}^{m\mu}(\bs b)h^{(1)}_{\nu}(kr_1)Y_{\nu}^{\mu}(\theta_1,\varphi_1),
	\end{equation}
	for $r_1>b$, where $\widehat{S}_{n\nu}^{m\mu}(\bs b)$ is given by \eqref{spmhat3d} and
	\begin{equation}\label{singularseparation}
	\begin{split}
	{S}_{n\nu}^{m\mu}(\bs b)
	&=4\pi(-1)^m\ri^{\nu-n}\sum\limits_{q=0}^{\infty}\ri^qh^{(1)}_q(kb)\overline{Y_q^{\mu-m}(\alpha,\beta)}\mathcal{G}(n,m;\nu,-\mu;q),
	\end{split}
	\end{equation}
	and $\mathcal{G}(n,m;\nu,-\mu;q)$ is a Gaunt coefficient.
\end{theorem}

The Gaunt coefficient $\mathcal{G}(n,m;\nu,-\mu;q)$ is defined using the Wigner $3-j$ symbol as follows
\begin{equation}
\mathcal{G}(n,m;\nu,\mu;q)=(-1)^{m+\mu}\mathcal{S}\begin{pmatrix}
n & \nu & q\\
0 & 0   & 0
\end{pmatrix}\begin{pmatrix}
n & \nu & q\\
m & \mu & -m-\mu
\end{pmatrix},
\end{equation}
where $\mathcal{S}=\sqrt{(2n+1)(2\nu+1)(2q+1)/(4\pi)}$. Although we have explicit formulas \eqref{spmhat3d} and \eqref{singularseparation} for the separation matrices $\widehat{S}_{n\nu}^{m\mu}(\bs b)$ and ${S}_{n\nu}^{m\mu}(\bs b)$, they are too complicated to be used directly for practical computations. Fortunately, recurrence formulas are available for their computations (cf. \cite{chew1992recurrence,gumerov2004recursions}).

%

\begin{figure}[ht!]\label{3dspherical}
	\centering
	\includegraphics[scale=1.1]{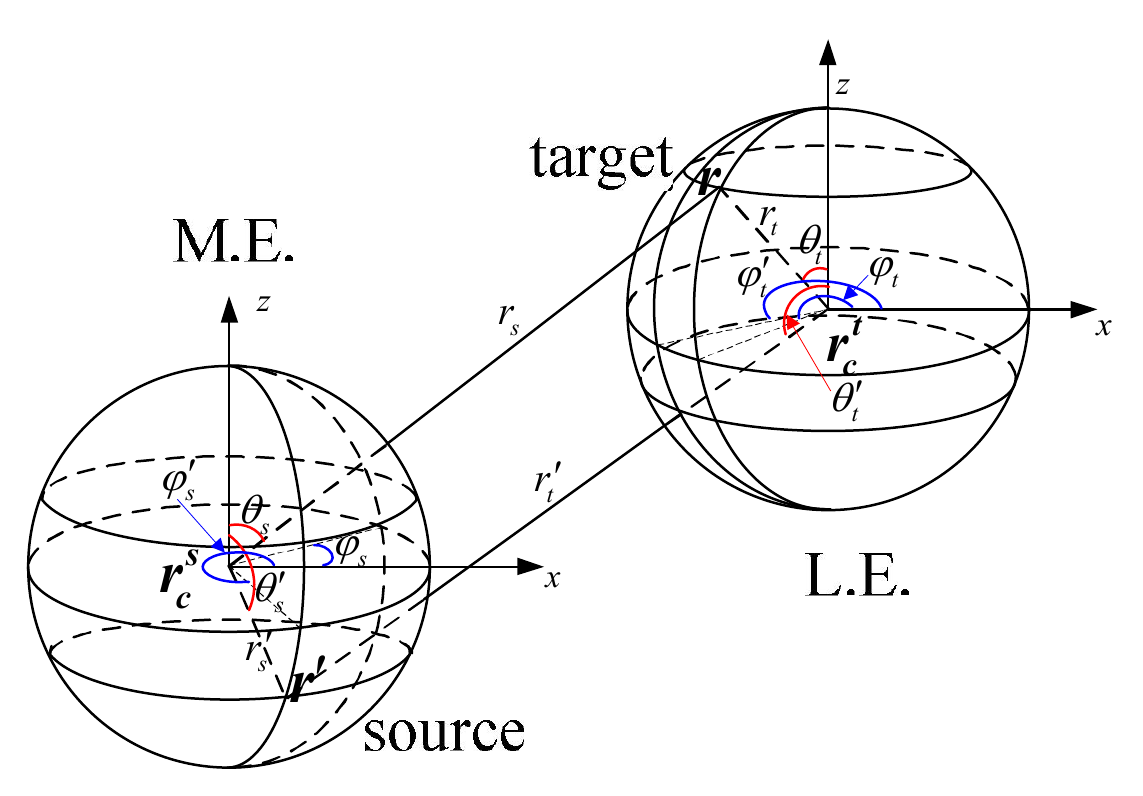}
	\caption{Spherical coordinates used in multipole and local expansions.}
\end{figure}
With these addition theorems, we can present the $h$- and $j$-expansions, which are also named as  multipole and local expansions in the FMM. Consider the free space Green's function of the Helmholtz equation with a source and a target at $\bs r'=(x', y', z')$ and $\bs r=(x, y, z)$, respectively. By using the addition Theorem \ref{zeroorderaddthm}, we have an $h$-expansion (multipole expansion) with respect to { a (source) center $\bs r_c^{s}$}:
\begin{equation}\label{freespace3dmulexp}
h^{(1)}_0(k|\bs r-\bs r'|)=\sum\limits_{|n|=0}^{\infty}\sum\limits_{m=-n}^nM_{nm}h_n^{(1)}(kr_s)Y_n^m(\theta_s,\varphi_s),
\end{equation}
and a $j$-expansion (local expansion) with respect to { a (target) center $\bs r_c^{t}$}:
\begin{equation}\label{freespace3dlocexp}
h^{(1)}_0(k|\bs r-\bs r'|)=\sum\limits_{|n|=0}^{\infty}\sum\limits_{m=-n}^nL_{nm}j_n(kr_t)Y_n^m(\theta_t,\varphi_t),
\end{equation}
where
\begin{equation}\label{melecoeffree}
M_{nm}=4\pi j_n(kr_s')\overline{Y_n^{m}(\theta'_s,\varphi_s')},\quad L_{nm}=4\pi h^{(1)}_n(kr_t')\overline{Y_n^{m}(\theta_t',\varphi_t')},
\end{equation}
{ $\bs r_c^{s}=(x_c^s, y_c^s, z_c^s)$ is the source center close to $\bs r'$, $\bs r_c^{t}=(x_c^t, y_c^t, z_c^t)$ is the target center close to $\bs r$, $(r_s, \theta_s,\varphi_s)$, $(r_t, \theta_t,\varphi_t)$
are the spherical coordinates of $\bs r-\bs r_c^s$ and $\bs r-\bs r_c^t$, respectively, and $(r_s',\theta_s',\varphi_s')$, $(r_t',\theta_t',\varphi_t')$
are the spherical coordinates of $\bs r'-\bs r_c^s$ and $\bs r'-\bs r_c^t$, respectively (Fig. \ref{3dspherical}).}

Applying addition Theorem \ref{thmhlynm} to $h_n^{(1)}(kr_s)Y_n^m(\theta_s, \varphi_s)$ in \eqref{freespace3dmulexp}, the translation from the $h$-expansion \eqref{freespace3dmulexp} to the $j$-expansion \eqref{freespace3dlocexp} is given by
\begin{equation}\label{metole}
L_{nm}=\sum\limits_{|\nu|=0}^{\infty}\sum\limits_{\mu=-\nu}^{\nu}{S}_{n\nu}^{m\mu}(\bs r_c^t-\bs r_c^s)M_{\nu\mu}.
\end{equation}
Similarly, we can shift the centers of $h$- and $j$-expansions via the following translations,
\begin{equation}\label{metome}
\begin{split}
\tilde{M}_{nm}
=&\sum\limits_{\nu=0}^{\infty}\sum\limits_{\mu=-\nu}^{\nu}\overline{\widehat{S}_{n\nu}^{m\mu}(\bs r_c^s-\tilde{\bs r}_c^s)}M_{\nu\mu},\quad \tilde L_{nm}=\sum\limits_{n=0}^{\infty}\sum\limits_{\mu=-\nu}^{\nu} \widehat{S}_{\nu n}^{\mu m}(\tilde{\bs r}_c^t-\bs r_c^t)L_{\nu\mu},
\end{split}
\end{equation}
where
\begin{equation}
\tilde{M}_{nm}=4\pi j_n(k\tilde r_s')\overline{Y_{n}^{m}(\tilde\theta'_s,\tilde\varphi_s')},\quad \tilde L_{nm}=4\pi h^{(1)}_{n}(k\tilde r_t')\overline{Y_{n}^{m}(\tilde\theta_t',\tilde\varphi_t')}
\end{equation}
are the $h$- and $j$-expansion coefficients at new centers $\tilde{\bs r}_c^s$ and $\tilde{\bs r}_c^t$, respectively.

An important feature in \eqref{freespace3dmulexp}-\eqref{freespace3dlocexp} is that the source and target coordinates { are separated,
which will be important in deriving low rank approximations for far fields in the FMM (cf. \cite{greengard1987fast,greengard1997new})}.
However, this target-source separation can also be achieved in the Fourier spectral domain.
A new derivation via Fourier domain for \eqref{freespace3dmulexp} and \eqref{freespace3dlocexp} will be given by using
an integral representation of $h_0^{(1)}(k|\bs r|)$.  Moreover, this approach can also be used to derive multipole and local
expansions for the reaction components of Green's functions in layered media later.

\subsection{A new derivation of $h$- and $j$-expansions}
For a spherical wave, we have the well-known Sommerfeld identity
\begin{equation}\label{sommerfeldid}
h_0^{(1)}(k|\bs r|)=\frac{1}{2k\pi}\int_0^{\infty}\int_0^{2\pi}k_{\rho}e^{\ri k_{\rho}(x\cos\alpha+y\sin\alpha)}\frac{e^{\ri k_z|z|}}{k_z}d\alpha  dk_{\rho},
\end{equation}
where $k_z=\sqrt{k^2-k_{\rho}^2}$.
It is necessary to point out that the identity \eqref{sommerfeldid} is for a lossless medium (i.e., $k$ is real) and needs to be considered as a limiting case of a lossy medium. The contour for the integral of $k_{\rho}$ is usually set by deforming the real axis to the fourth quadrant. In order to satisfy the outgoing wave radiation condition in the integrand, we have to ensure that $\mathfrak{Re}k_{z}>0$ and $\mathfrak{Im}k_{z}>0$.
With the Sommerfeld identity \eqref{sommerfeldid}, we can make a source-target separation in the spectral domain as follows
\begin{equation}
\label{positivecase}
\begin{split}
h_0^{(1)}(k|\bs r-\bs r'|)=\frac{1}{2k\pi}\int_0^{\infty}\int_0^{2\pi}k_{\rho}\frac{e^{\ri \bs k\cdot(\bs r-\bs r_c^s)}e^{-\ri \bs k\cdot(\bs r'-\bs r_c^s)}}{k_z}d\alpha  dk_{\rho},\\
h_0^{(1)}(k|\bs r-\bs r'|)=\frac{1}{2k\pi}\int_0^{\infty}\int_0^{2\pi}k_{\rho}\frac{e^{\ri \bs k\cdot(\bs r-\bs r_c^t)}e^{-\ri \bs k\cdot(\bs r'-\bs r^t_c)}}{k_z}d\alpha  dk_{\rho}
\end{split}
\end{equation}
for $z\geq z'$ where $\bs k=(k_{\rho}\cos\alpha, k_{\rho}\sin\alpha, k_z).$
Without loss of generality, here we only consider the case $z\geq z'$ for an illustration.

Next, we evoke the well-known Funk-Hecke formula \cite{watson,martin2006multiple}
\begin{equation}\label{Funk-Hecke-formula}
\begin{split}
e^{\ri \bs k\cdot{\bs r}}=&\sum\limits_{n=0}^{\infty}\sum\limits_{m=-n}^n4\pi \ri^nj_n(kr)\overline{Y_n^m(\theta,\varphi)}\widehat{P}_n^m\Big(\frac{k_z}{k}\Big)e^{\ri m\alpha}\\
=&\sum\limits_{n=0}^{\infty}\sum\limits_{m=-n}^n4\pi \ri^nj_n(kr)Y_n^m(\theta,\varphi)\widehat{P}_n^m\Big(\frac{k_z}{k}\Big)e^{-\ri m\alpha},
\end{split}
\end{equation}
which gives spherical harmonic expansions for plane waves and will be used for the plane waves inside \eqref{positivecase}. This classic Funk-Hecke formula only works for propagating plane wave, i.e., $0<k_{\rho}\leq k$ where all components of the propagation vector $\bs k$ are real. However, the Fourier spectral representations \eqref{positivecase} involve not only propagating but also evanescent plane waves, i.e., $k_{\rho}>k$ where $k_z$ is purely imaginary. Therefore, for the spherical harmonic expansion of the evanescent plane wave, we need to extend the range of Funk-Hecke formula from $k_z\in \{\omega| 0\leq \omega\leq k\}$ to $k_z\in\{\omega| 0\leq \omega\leq k\}\cup \{\omega|\omega=\ri y, \, y\geq 0\}$. Fortunately, we can show that formula \eqref{Funk-Hecke-formula} holds for all $k_{z}\in \mathbb C$ by choosing an appropriate branch for the square root function. To see this, the left hand side of \eqref{Funk-Hecke-formula} can be written as
\begin{equation}\label{lefthandside}
e^{\ri \bs k\cdot{\bs r}}=e^{\ri(\sqrt{k^2-k_z^2}(x\cos\alpha+y\sin\alpha)+k_zz)},
\end{equation}
while the right hand side only has $k_z$ as the argument of the normalized associated Legendre function. For the normalized associated Legendre function  $\widehat{P}_n^m(x)$, we have $\widehat{P}_n^{-m}(x)=(-1)^m\widehat{P}_n^m(x), m>0$ and also the Rodrigues' formula
\begin{equation}
\widehat P_n^m(x)=\frac{1}{2^nn!}\sqrt{\frac{2n+1}{4\pi}\frac{(n-m)!}{(n+m)!}}(1-x^2)^{\frac{m}{2}}\frac{d^{n+m}}{dx^{n+m}}(x^2-1)^n.
\end{equation}
If $m$ is even, the analytic extension of $\widehat P_n^m(x)$ into the complex plane defined by
\begin{equation}
\widehat P_n^m(z)=\frac{1}{2^nn!}\sqrt{\frac{2n+1}{4\pi}\frac{(n-m)!}{(n+m)!}}(1-z^2)^{\frac{m}{2}}\frac{d^{n+m}}{dz^{n+m}}(z^2-1)^n,\quad z\in \mathbb C,
\end{equation}
is a polynomial of $z$ and hence is an entire function. For odd $m$, $\widehat P_n^m(x)$ can be written as
\begin{equation}
\widehat P_n^m(x)=\sqrt{1-x^2}Q_{n-1}(x),
\end{equation}
where $Q_{n-1}(x)$ is a polynomial of at most $(n-1)$-th degree. Therefore, a complex extension defined by
\begin{equation}\label{extensionmodd}
\widehat P_n^m(z)=\sqrt{1-z^2}Q_{n-1}(z),\forall z\in\mathbb C,
\end{equation}
involves the multivalued function $\sqrt{1-z^2}$ similarly as in \eqref{lefthandside}. In order to have an analytic continuation for the Funk-Hecke formula \eqref{Funk-Hecke-formula} for complex $k_z$, we need to choose an appropriate branch in \eqref{lefthandside} and \eqref{extensionmodd}. Let us consider polar forms
$$z+1=r_1e^{\ri\theta_1}, \quad -\pi<\theta_1\leq\pi,\quad  z-1=r_2e^{\ri\theta_2},\quad -\pi<\theta_2\leq\pi,$$
where $\theta_1$ and $\theta_2$ are the principal values of the arguments of complex numbers $z+1$ and $z-1$.
Then, we cut the complex plane from $-1$ to $+1$ along the real axis and choose the branch
\begin{equation}\label{branch}
\sqrt{1-z^2}=-\ri \sqrt{r_1r_2}e^{\ri\frac{\theta_1+\theta_2}{2}}.
\end{equation}
For any $x\in[-1, 1]$, using this branch leads to
\begin{equation}
\lim\limits_{\epsilon\rightarrow 0^+}\sqrt{1-(x+\ri\epsilon)^2}=\sqrt{1-x^2},\quad \lim\limits_{\epsilon\rightarrow 0^-}\sqrt{1-(x+\ri\epsilon)^2}=-\sqrt{1-x^2}.
\end{equation}
The extensions $e^{\ri \bs k\cdot\bs r}$ and $\widehat P_n^m(z)$ defined by using \eqref{branch} enjoy the same property, e.g.,
\begin{equation}
\lim\limits_{\epsilon\rightarrow 0^+}\widehat P_n^m(x+\ri\epsilon)=\widehat P_n^m(x),\quad \lim\limits_{\epsilon\rightarrow 0^-}\widehat P_n^m(x+\ri\epsilon)=-\widehat P_n^m(x),\quad\forall x\in[-1, 1].
\end{equation}
We note that the extension of associated Legendre function also satisfies
\begin{equation}
\widehat P_n^m(-z)=(-1)^{n+m}\widehat P_n^m(z),  \quad \widehat P_n^{-m}(z)=(-1)^{m}\widehat P_n^m(z),
\end{equation}
and the recurrence formulas
\begin{equation}
\begin{split}
\widehat P_0^0(z)=\frac{1}{4\pi},\quad \widehat P_n^n(z)=-\sqrt{\frac{2n+1}{2n}}\ri\sqrt{r_1r_2}e^{\ri\frac{\theta_1+\theta_2}{2}}\widehat P_{n-1}^{n-1}(z),\quad n=1, 2, \cdots,\\
\widehat P_{m+k}^m(z)=a_{m+k}^mz\widehat P_{m+k-1}^{m}(z)-b_{m+k}^m\widehat P_{m+k-2}^{m}(z),\quad m=0, 1, \cdots,\;k=1, 2,\cdots,
\end{split}
\end{equation}
where
$$a_{n}^m=\sqrt{\frac{(2n-1)(2n+1)}{(n-m)(n+m)}},\quad b_{n}^m=\sqrt{\frac{(2n+1)(n-m-1)(n+m-1)}{(2n-3)(n+m)(n-m)}}.$$
The above recurrence formulas will be used in the implementation of the FMM.

Next, for the proof of the Funk-Hecke formula in the complex plane,  we need the following lemmas.
\begin{lemma}\label{lemma1}
	For any real number $a\geq 0$, there holds
	\begin{equation}\label{planewaveexp}
	e^{\ri a z}=\sum\limits_{n=0}^{\infty}(2n+1)\ri^nj_n(a)P_n(z), \quad\forall z\in\mathbb C,
	\end{equation}
	where $j_n(a)=\sqrt{\frac{\pi}{2a}}J_{n+1/2}(a)$ is the spherical Bessel function of the first kind, $P_n(z)$ is the Legendre polynomial extended to the complex plane.
\end{lemma}
\begin{proof}
	Recall the series (cf. \cite[10.60.7]{Olver2010})
	\begin{equation}
	e^{\ri a\cos\theta}=\sum\limits_{n=0}^{\infty}(2n+1)\ri^nj_n(a)P_n(\cos\theta),
	\end{equation}
	we can see that \eqref{planewaveexp} holds for all $z\in [-1, 1]$. Next, we consider its extension to the whole complex plane. Apparently, $e^{\ri a z}$ is an entire function of $z$. Meanwhile, the spherical Bessel function $j_n(a)$ has the following upper bound (cf. \cite[9.1.62]{Abr.I64})
	\begin{equation}
	|j_n(a)|\leq \frac{\Gamma(\frac{3}{2})}{\Gamma(n+\frac{3}{2})}\Big(\frac{a}{2}\Big)^n\leq \frac{1}{n!}\Big(\frac{a}{2}\Big)^n.
	\end{equation}
	The Legendre polynomial $P_n(z)$ is a polynomial of degree $n$ with $n$ distinct roots $\{z_j\}_{j=1}^n$ in the interval $[-1, 1]$. Therefore,
	\begin{equation}
	|P_n(z)|=|a_n|\prod\limits_{j=1}^n|z-z_j|\leq 2^n(|z|+1)^n, \quad \forall z\in\mathbb C,
	\end{equation}
	here the estimate $a_n=\frac{(2n)!}{2^n(n!)^2}\leq 2^n$ for the coefficient of $z^n$ is used. These upper bounds for $j_n(a)$ and $P_n(z)$ give an estimate
	\begin{equation}
	\sum\limits_{n=0}^{\infty}(2n+1)|\ri^nj_n(a)P_n(z)|\leq \sum\limits_{n=0}^{\infty}(2n+1)\frac{a^n(|z|+1)^n}{n!}=(2a(|z|+1)+1)e^{a(|z|+1)}.
	\end{equation}
	It is easy to show that the series on the righthand side of \eqref{planewaveexp} converges uniformly in any compact set $D\subset \mathbb C$ and hence converges to an entire function of $z$. By the analytic extension theory, we have the proof.
\end{proof}

By using the branch defined in \eqref{branch} for the square roots, we have the extension of the well-known Legendre addition theorem \cite[p.395]{Whi.W27}.
\begin{lemma}\label{lemma2}
	Let $\bs w=(\sqrt{1-w^2}\cos\alpha, \sqrt{1-w^2}\sin\alpha, w)$ be a vector with complex entries, $\theta, \varphi$ be the azimuthal angle and polar angles of a unit vector $\hat{\bs r}$. Define
	\begin{equation}
	\beta(w)=w\cos\theta+\sqrt{1-w^2}\sin\theta\cos(\alpha-\varphi),
	\end{equation}
	then
	\begin{equation}\label{legendreadd}
	P_n(\beta(w))=\frac{4\pi}{2n+1}\sum\limits_{m=-n}^n\widehat P_n^m(\cos\theta)\widehat P_n^m(w)e^{\ri m(\alpha-\varphi)},
	\end{equation}
	for all $w\in\mathbb C$.
\end{lemma}
\begin{proof}
	If $w\in [-1, 1]$ is real, we have $\beta(w)=\cos\langle\bs w, \hat{\bs r}\rangle$ and \eqref{legendreadd} is the well-known Legendre addition theorem. Define
	\begin{equation}
	g_n(w)=\frac{4\pi}{2n+1}\sum\limits_{m=-n}^n\widehat P_n^m(\cos\theta)\widehat P_n^m(w)e^{\ri m(\alpha-\varphi)}.
	\end{equation}
	By using the branch defined in \eqref{branch} for the square roots in $\beta(w)$ and $\widehat P_n^m(w)$, $P_n(\beta(w))$ and $g_n(w)$ are extended to the complex plane. Moreover, they are analytic in $\mathbb C\setminus [-1, 1]$ and satisfy limit properties
	\begin{equation*}
	\begin{split}
	\lim\limits_{\epsilon\rightarrow 0^+}P_n(\beta(x+\ri\epsilon))= P_n(\beta(x)),\; \lim\limits_{\epsilon\rightarrow 0^-}=P_n(\beta(x+\ri\epsilon))=-P_n(\beta(x)),\\
	\lim\limits_{\epsilon\rightarrow 0^+}g_n(x+\ri\epsilon)=g_n(x),\quad \lim\limits_{\epsilon\rightarrow 0^-}g_n(x+\ri\epsilon)=-g_n(x)
	\end{split}
	\end{equation*}
	for all $x\in[-1, 1]$. That implies $P_n(\beta(w))-g_n(w)$ takes angular boundary values zero on the set $[-1, 1]$ (taking limit from the upper complex plane). By using Luzin-Privalov theorem \cite{collingwood2004theory}, we conclude that $P_n(\beta(w))-g_n(w)=0$ in $\mathbb C\setminus [-1, 1]$.
\end{proof}

\begin{proposition}\label{prop:Funk-Hecke}
	Given $\bs r=(x, y, z)\in \mathbb R^3$, $k>0$, $\alpha\in[0, 2\pi)$ and denoted by $(r,\theta,\varphi)$ the spherical coordinates of $\bs r$, $\bs k=(\sqrt{k^2-k_z^2}\cos\alpha, \sqrt{k^2-k_z^2}\sin\alpha, k_z)$ is a vector of complex entries. Choosing the branch \eqref{branch} for $\sqrt{k^2-k_z^2}$ in $e^{\ri \bs k\cdot{\bs r}}$ and $\widehat P_n^m(\frac{k_z}{k})$, then
	\begin{equation}\label{extfunkhecke}
	e^{\ri \bs k\cdot{\bs r}}=\sum\limits_{n=0}^{\infty}\sum\limits_{m=-n}^n \overline{A_{n}^m(\bs r)}\ri^n\widehat{P}_n^m\Big(\frac{k_z}{k}\Big)e^{\ri m\alpha}=\sum\limits_{n=0}^{\infty}\sum\limits_{m=-n}^n A_{n}^m(\bs r)\ri^n\widehat{P}_n^m\Big(\frac{k_z}{k}\Big)e^{-\ri m\alpha},
	\end{equation}
	holds for all $k_z\in\mathbb C$, where
	$$A_{n}^m(\bs r)=4\pi j_n(kr)Y_n^m(\theta,\varphi).$$
\end{proposition}
\begin{proof}
	Define
	$$\beta=\frac{k_z}{k}\cos\theta+\sqrt{1-\frac{k_z^2}{k^2}}\sin\theta\cos(\alpha-\varphi),$$
	then, $kr\beta=\bs k\cdot\bs r$. Let $a=kr$,  $z=\beta$ in \eqref{planewaveexp}, we have
	\begin{equation}
	e^{\ri \bs k\cdot\bs r}=\sum\limits_{n=0}^{\infty}(2n+1)\ri^nj_n(kr)P_n(\beta).
	\end{equation}
	Finally, the extension of Funk-Hecke formula \eqref{extfunkhecke} follows by applying Lemma \ref{lemma2}.
\end{proof}

{
To show the convergence rate of the expansion \eqref{extfunkhecke}, we consider three different cases with $\bs r=(0.3, 0.4, z)$, $\alpha=\frac{\pi}{4}$.
The convergence rates against $p$ are depicted in Fig. \ref{FunkHeckeApperr} for the relative  $\ell_2$ error. Spectral convergence rates against $p$
are observed while the expansion converges more slowly if one of $|\bs k|$, $|\bs r|$ and $|k_z|$ gets larger.
\begin{figure}[ht!]
	\centering
	\subfigure[$z=0.5$, $k_z=1+\ri$]{\includegraphics[scale=0.23]{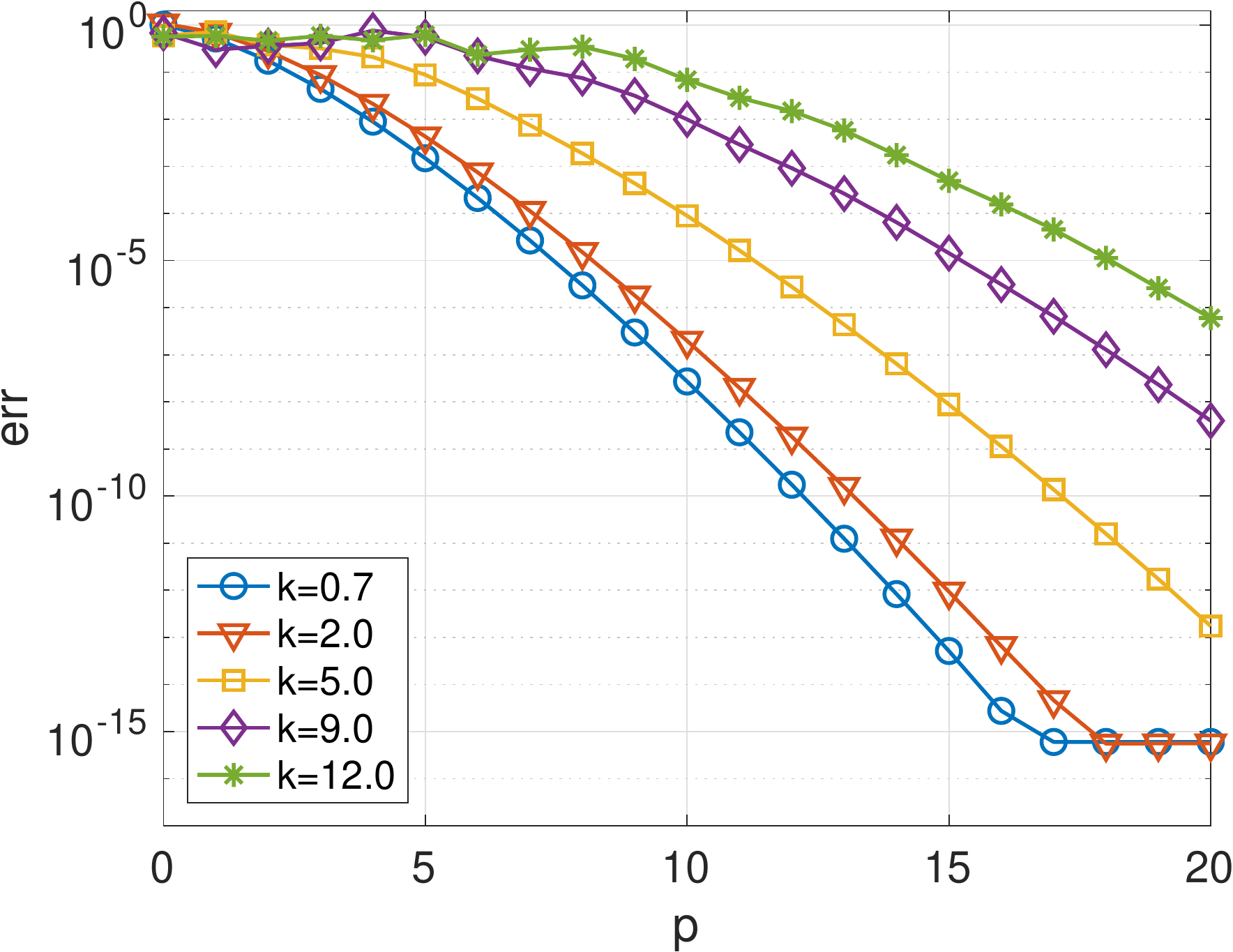}}
	\subfigure[$k=1.8$, $k_z=1-1.3\ri$]{\includegraphics[scale=0.23]{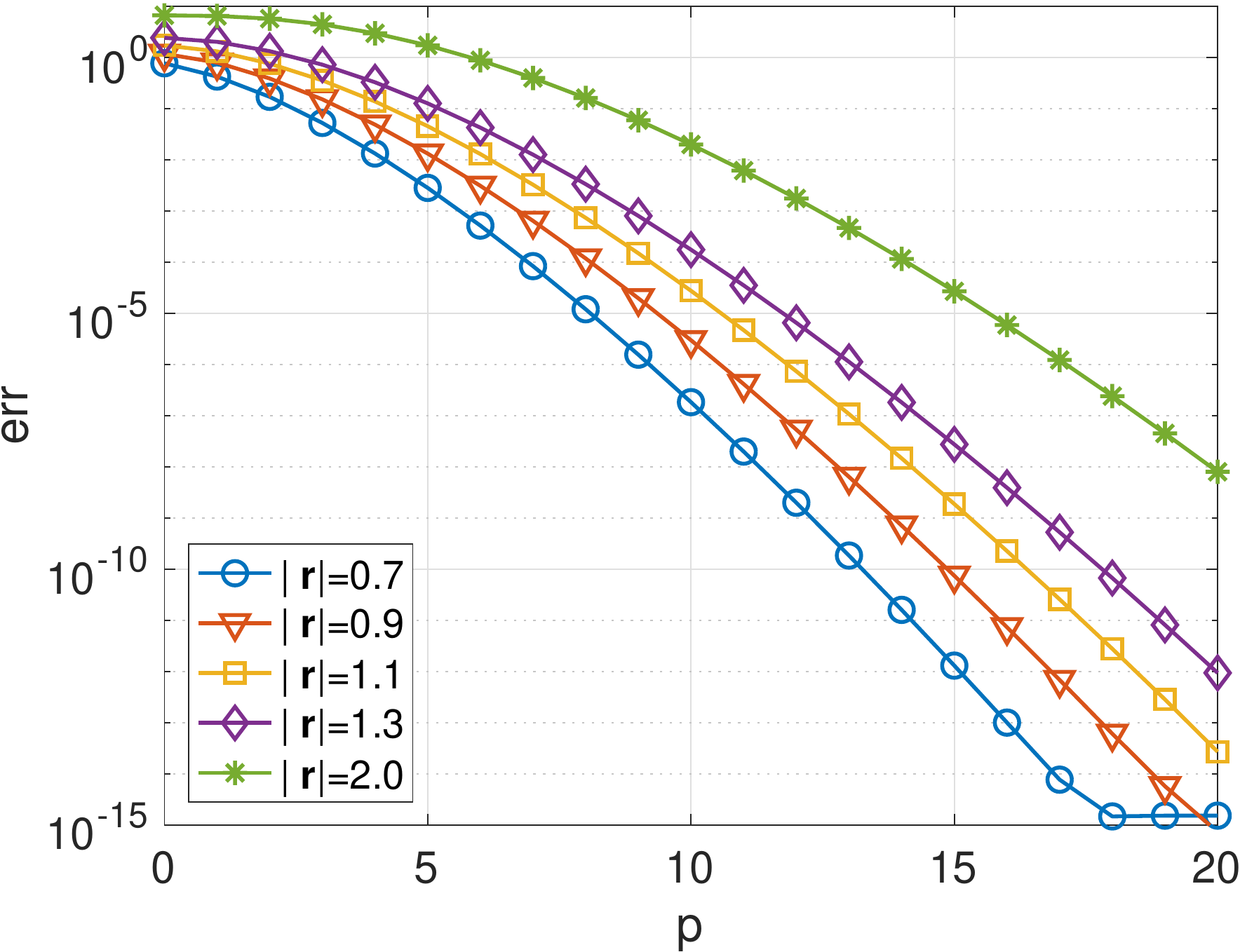}}
	\subfigure[$z=0.5$, $k=1.8$]{\includegraphics[scale=0.23]{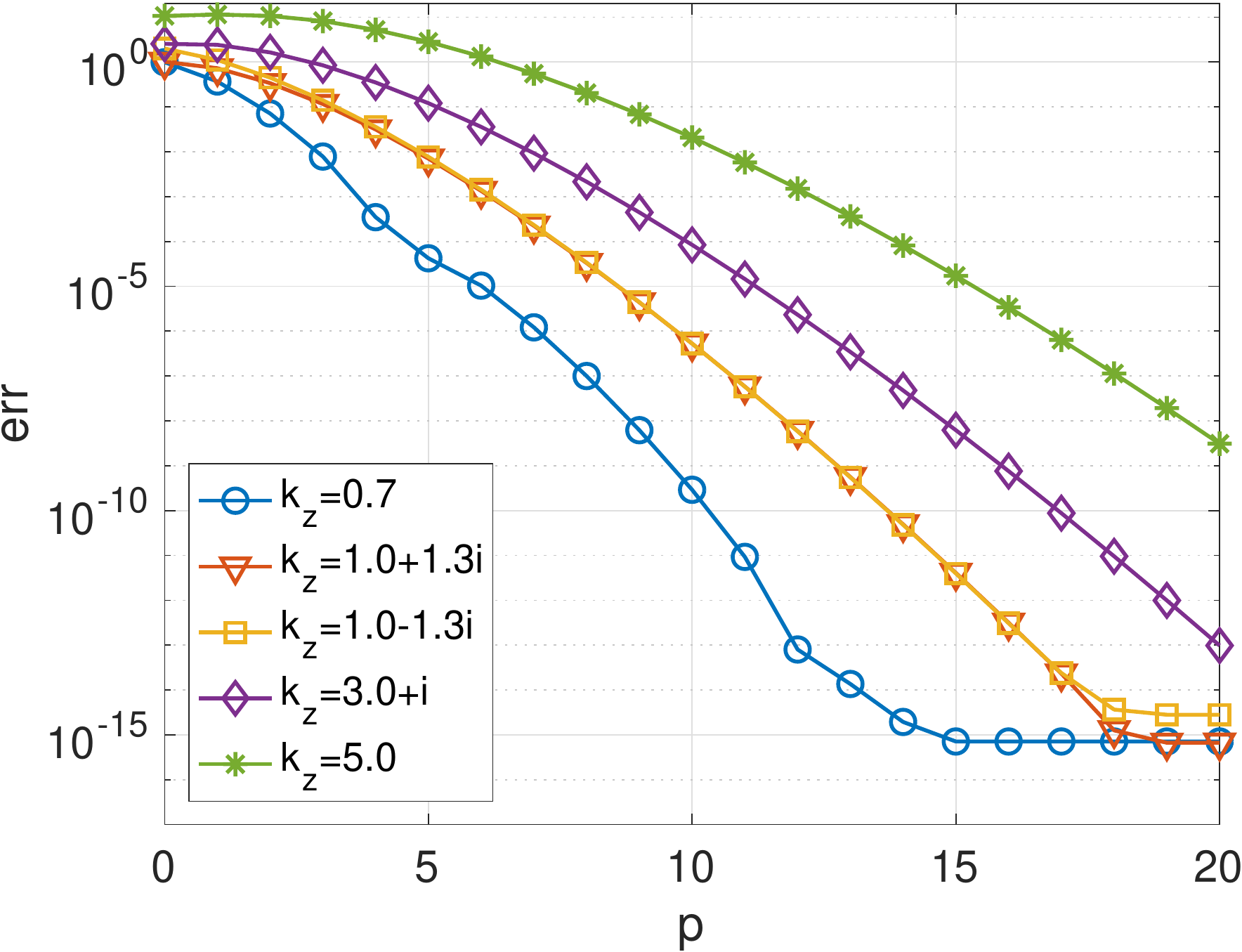}}
	\caption{Convergence rate of the approximation using extended Funk-Hecke formula.}
	\label{FunkHeckeApperr}%
\end{figure}
}

Applying the spherical harmonic expansion \eqref{extfunkhecke} to exponential functions $e^{-\ri \bs k\cdot(\bs r'-\bs r_c^s)}$ and $e^{\ri \bs k\cdot(\bs r-\bs r_c^t)}$ in \eqref{positivecase} gives
\begin{equation}\label{meinspectraldomain}
h^{(1)}_0(k|\bs r-\bs r'|)=\sum\limits_{n=0}^{\infty}\sum\limits_{m=-n}^{n}\frac{M_{nm}}{2k\pi}\int_0^{\infty}\int_0^{2\pi}k_{\rho}\frac{e^{\ri \bs k\cdot(\bs r-\bs r_c^s)}}{ k_z}(-\ri)^n\widehat{P}_n^m\Big(\frac{k_z}{k}\Big)e^{\ri m\alpha}d\alpha  dk_{\rho},
\end{equation}
and
\begin{equation}\label{leinspectraldomain}
h^{(1)}_0(k|\bs r-\bs r'|)=\sum\limits_{n=0}^{\infty}\sum\limits_{m=-n}^{n}\hat L_{nm}j_n(kr_t)Y_n^m(\theta_t,\varphi_t),
\end{equation}
for $z\geq z'$, where $M_{nm}$ is defined in \eqref{melecoeffree} and
\begin{equation}
\hat L_{nm}=\frac{2\ri^n}{k}\int_0^{\infty}\int_0^{2\pi}k_{\rho}\frac{e^{\ri \bs k\cdot(\bs r_c^t-\bs r')}}{ k_z}\widehat{P}_n^m\Big(\frac{k_z}{k}\Big)e^{-\ri m\alpha}d\alpha  dk_{\rho}.
\end{equation}
For the convergence of the Sommerfeld-type integrals in the above expansions, we only consider centers such that their $z$-coordinates satisfy $z_c^s<z$ and $z_c^t>z'$.
Recall the identity
\begin{equation}\label{wavefunspectralform}
h_n^{(1)}(k|\bs r|)Y_n^m(\theta,\varphi)=\frac{1}{2k\pi}\int_0^{\infty}\int_0^{2\pi}k_{\rho}\frac{e^{\ri \bs k\cdot\bs r}}{ k_z}(-\ri)^n\widehat P_n^{m}\Big(\frac{k_z}{k}\Big)e^{\ri m\alpha} d\alpha dk_{\rho}
\end{equation}
for $z\geq 0$, we see that \eqref{meinspectraldomain} and \eqref{leinspectraldomain} are exactly the $h$-expansion \eqref{freespace3dmulexp} and $j$-expansion \eqref{freespace3dlocexp} for the case of $z\geq z'$.

To derive the translation from the $h$-expansion \eqref{meinspectraldomain} to the $j$-expansion \eqref{leinspectraldomain}, we perform a further splitting in \eqref{meinspectraldomain}
\begin{equation}
e^{\ri \bs k\cdot(\bs r-\bs r_c^s)}=e^{\ri \bs k\cdot(\bs r-\bs r_c^t)}e^{\ri \bs k\cdot(\bs r_c^t-\bs r_c^s)}
\end{equation}
 and apply expansion \eqref{extfunkhecke} to obtain the following translation
\begin{equation*}
\begin{split}
L_{nm}=&\frac{2}{ k}\sum\limits_{\nu=0}^{\infty}\sum\limits_{\mu=-\nu}^{\nu}M_{\nu\mu}\int_0^{\infty}\int_0^{2\pi}k_{\rho}\frac{e^{\ri \bs k(\bs r_c^t-\bs r_c^s)}}{k_z} (-1)^{\nu}\ri^{n+\nu}\widehat P_{\nu}^{\mu}\Big(\frac{k_z}{k}\Big)\widehat P_{n}^{m}\Big(\frac{k_z}{k}\Big)e^{\ri(\mu-m)\alpha}d\alpha  dk_{\rho},
\end{split}
\end{equation*}
which implies an integral representation of $S_{n\nu}^{m\mu}(\bs r_c^t-\bs r_c^s)$ in \eqref{metole}. {In order to ensure the convergence of the Sommerfeld-type integral in the translation operator, the $z$-coordinates of the centers are also assumed to satisfy $z_c^t>z_c^s$. }

\section{FMM for 3-D Helmholtz equation in layered media}
In this section, the multipole and local expansions for the reaction components of layered media Green's function of 3-D Helmholtz equation will be derived by using the techniques introduced in the last section. Based on these expansions and relevant translation operators, FMM for 3-D Helmholtz Green's function in layered media can be developed.

\subsection{Green's function of Helmholtz equation in layered media} {Let us first review the integral representation of the layered Green's function derived in \cite{wang2018}}. Consider a layered medium consisting of $L$ interfaces located at $z=d_{\ell
},\ell=0,1,\cdots,L-1$ as shown in Fig. \ref{layerstructure}.
\begin{figure}[ht!]\label{layerstructure}
	\centering
	\includegraphics[scale=0.8]{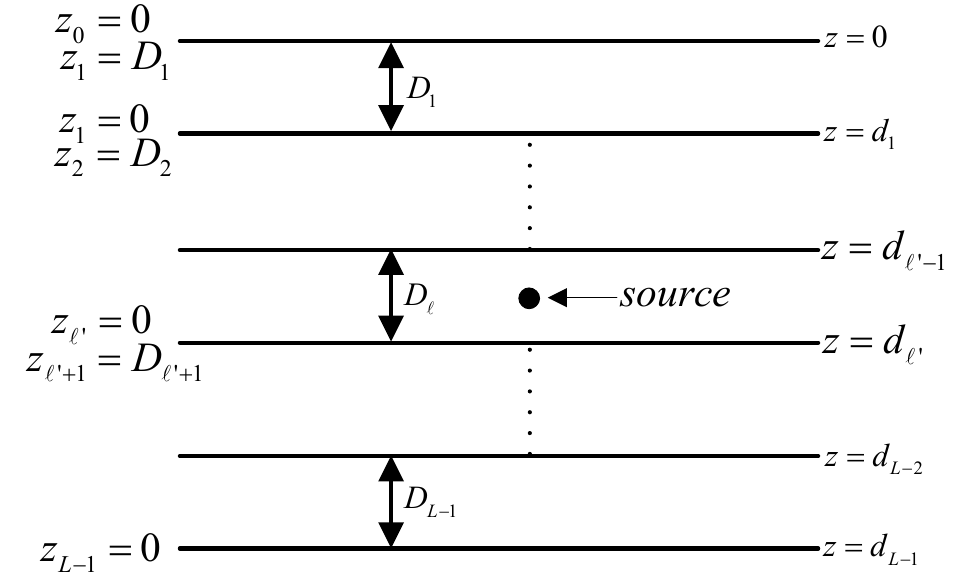}
	\caption{Sketch of the layer structure for general multi-layer media.}
\end{figure}
A point
source at $\boldsymbol{r}^{\prime}=(x^{\prime},y^{\prime},z^{\prime})$ is located in the
$\ell^{\prime}$-th layer ($d_{\ell^{\prime}}<z^{\prime}<d_{\ell^{\prime}-1}$).
The layered Green's function for the Helmholtz equation satisfies
\begin{equation}
\boldsymbol{\Delta}u(\boldsymbol{r},\boldsymbol{r}^{\prime
})+k_{\ell}^{2}u(\boldsymbol{r},\boldsymbol{r}^{\prime
})=-\delta(\boldsymbol{r},\boldsymbol{r}^{\prime}),
\end{equation}
at a field point $\boldsymbol{r}=(x,y,z)$ in the $\ell$-th layer ($d_{\ell
}<z<d_{\ell}-1$) where $\delta(\boldsymbol{r},\boldsymbol{r}^{\prime})$ is the
Dirac delta function and $k_{\ell}$ is the wave number in the $\ell$-th layer.
The system can be solved analytically in the Fourier $(k_x, k_y)$-domain for each layer in $z$ by imposing
transmission conditions at the interface between $\ell$-th and $(\ell-1)$-th
layer ($z=d_{\ell-1})$ as well as the decay conditions in the top and bottom-most layers as
$z\rightarrow\pm\infty$ \cite{cho2012parallel}.
Inside a given layer, say, the $\ell$-th layer, the solution has the form
\begin{equation}\label{domaingreenfn}
u(\bs r,\bs r')=\begin{cases}
\displaystyle u_{\ell'\ell'}^{\rm r}(\bs r, \bs r')+\frac{e^{\ri k_{\ell'}|\bs r-\bs r'|}}{4\pi|\bs r-\bs r'|}, & d_{\ell'}<z<d_{\ell'-1}, \\
\displaystyle u_{\ell\ell'}^{\rm r}(\bs r, \bs r'), & d_{\ell}<z<d_{\ell-1}\;{\rm and}\; \ell\neq \ell',
\end{cases}
\end{equation}
where
\begin{equation}\label{generalreact}
u_{\ell\ell'}^{\rm r}(\bs r, \bs r')=\begin{cases}
\displaystyle u_{0\ell'}^{\uparrow}(\bs r,\bs r'), & \ell =0,\\
\displaystyle u_{\ell\ell'}^{\uparrow}(\bs r,\bs r')+u_{\ell\ell'}^{\downarrow}(\bs r,\bs r'), & 0<\ell<L,\\
\displaystyle u_{L\ell'}^{\downarrow}(\bs r,\bs r'), &\ell=L,
\end{cases}
\end{equation}
is called the reaction field in the $\ell$-th layer due to wave reflections and transmissions by the layered media. We can see that the reaction field $u_{\ell\ell'}^{\rm r}(\bs r, \bs r')$ has up-going and down-going components inside intermediate layers $(0<\ell<L)$. Only up-going or down-going component is required in the top- and bottom-most layers, respectively. The up- and down-going components have Sommerfeld integral representations
\begin{equation}\label{greenfuncomponent}
\begin{split}
u_{\ell\ell'}^{\uparrow}(\bs r, \bs r')=&\frac{\ri}{8\pi^2 }\int_0^{\infty}\int_0^{2\pi}k_{\rho}e^{\ri\bs k_{\alpha}\cdot(\bs\rho-\bs\rho')}\frac{e^{\ri k_{\ell z} (z-d_{\ell})}}{k_{\ell z}}\psi_{\ell\ell'}^{\uparrow}(k_{\rho}, z')d\alpha dk_{\rho},\quad \ell<L,\\
 u_{\ell\ell'}^{\downarrow}(\bs r, \bs r')=&\frac{\ri}{8\pi^2}\int_0^{\infty}\int_0^{2\pi}k_{\rho}e^{\ri\bs k_{\alpha}\cdot(\bs\rho-\bs\rho')}\frac{e^{\ri k_{\ell z} (d_{\ell-1}-z)}}{k_{\ell z}}\psi_{\ell\ell'}^{\downarrow}(k_{\rho}, z')d\alpha dk_{\rho},\quad \ell>0,
\end{split}
\end{equation}
where $\bs k_{\alpha}=(k_{\rho} \cos\alpha,k_{\rho} \sin\alpha)=(k_x, k_y)$, $\bs\rho=(x, y)$, $\bs \rho'=(x', y')$, $k_{\ell z}=\sqrt{k_{\ell}^2-k_{\rho}^2}$,
\begin{equation}\label{totaldensity}
\begin{cases}
\displaystyle \psi_{\ell0}^{\uparrow}(k_{\rho}, z')=e^{\ri k_{0 z}(z'-d_0)}\sigma_{\ell\ell'}^{\uparrow\uparrow}(k_{\rho}),\\[8pt]
\displaystyle \psi_{\ell\ell'}^{\uparrow}(k_{\rho}, z')=e^{\ri k_{\ell' z}(z'-d_{\ell'})}\sigma_{\ell\ell'}^{\uparrow\uparrow}(k_{\rho})+e^{\ri k_{\ell' z}(d_{\ell'-1}-z')}\sigma_{\ell\ell'}^{\uparrow\downarrow}(k_{\rho}),\quad 0<\ell'<L,\\[8pt]
\displaystyle \psi_{\ell\ell'}^{\downarrow}(k_{\rho}, z')=e^{\ri k_{\ell' z}(z'-d_{\ell'})}\sigma_{\ell\ell'}^{\downarrow\uparrow}(k_{\rho})+e^{\ri k_{\ell' z}(d_{\ell'-1}-z')}\sigma_{\ell\ell'}^{\downarrow\downarrow}(k_{\rho}),\quad 0<\ell'<L,\\[8pt]
\displaystyle \psi_{\ell L}^{\downarrow}(k_{\rho}, z')=e^{\ri k_{\ell' z}(d_{L-1}-z')}\sigma_{\ell L}^{\downarrow\downarrow}(k_{\rho}).
\end{cases}
\end{equation}

The reaction densities $\sigma_{\ell\ell'}^{\uparrow\uparrow}(k_{\rho}), \sigma_{\ell\ell'}^{\uparrow\downarrow}(k_{\rho}), \sigma_{\ell\ell'}^{\downarrow\uparrow}(k_{\rho}), \sigma_{\ell\ell'}^{\downarrow\downarrow}(k_{\rho})$ only depend on the layer structure and wave numbers.
Equation \eqref{greenfuncomponent}-\eqref{totaldensity} are general formulas which are applicable to multi-layered media. Here, we give explicit formulas for the reaction densities of the three layers media with source in the middle layer, i.e.,
\begin{equation*}
\begin{split}
\sigma_{01}^{\uparrow\uparrow}(k_{\rho})&=\frac{k_1k_{0z}(k_1k_{1z}-k_2k_{2z})e^{-\ri d_1k_{1z}}}{k_0k_{0z}\kappa_{11}-\ri k_1k_{1z}\kappa_{12}},\quad \sigma_{01}^{\uparrow\downarrow}(k_{\rho})=\frac{k_1k_{0z}k_1k_{1z}+k_2k_{2z}}{k_0k_{0z}\kappa_{11}-\ri k_1k_{1z}\kappa_{12}},\\
\sigma_{11}^{\uparrow\uparrow}(k_{\rho})&=\frac{(k_1k_{1z}-k_2k_{2z})(k_1k_{1z}+k_0k_{0z})}{2(k_0k_{0z}\kappa_{11}-\ri k_1k_{1z}\kappa_{12})},\;\sigma_{11}^{\downarrow\downarrow}(k_{\rho})=\frac{(k_1k_{1z}-k_0k_{0z})(k_1k_{1z}+k_2k_{2z})}{2(k_0k_{0z}\kappa_{11}-\ri k_1k_{1z}\kappa_{12})},\\ \sigma_{11}^{\uparrow\downarrow}(k_{\rho})&=\frac{k_1k_{1z}-k_2k_{2z}}{k_0k_{0z}\kappa_{11}-\ri k_1k_{1z}\kappa_{12}}\frac{k_1k_{1z}-k_0k_{0z}}{2}e^{-\ri d_1k_{1z}},\\
\sigma_{11}^{\downarrow\uparrow}(k_{\rho})&=\frac{k_1k_{1z}-k_0k_{0z}}{k_0k_{0z}\kappa_{11}-\ri k_1k_{1z}\kappa_{12}}\frac{k_1k_{1z}-k_2k_{2z}}{2}e^{-\ri d_1k_{1z}}\\
\sigma_{21}^{\downarrow\uparrow}(k_{\rho})&=\frac{k_1 k_{2z}(k_1k_{1z}+k_0k_{0z})}{k_0k_{0z}\kappa_{11}-\ri k_1k_{1z}\kappa_{12}},\quad \sigma_{21}^{\downarrow\downarrow}(k_{\rho})=\frac{k_1 k_{2z}(k_1k_{1z}-k_0k_{0z})e^{-\ri d_1k_{1z}}}{k_0k_{0z}\kappa_{11}-\ri k_1k_{1z}\kappa_{12}}.
\end{split}
\end{equation*}
where
\begin{equation*}
\begin{split}
&\kappa_{11}=\frac{k_1k_{1z}-k_2k_{2z}}{2}e^{-2\ri d_1k_{1z}}+\frac{k_1k_{1z}+k_2k_{2z}}{2}, \\ &\kappa_{12}=\ri\Big(\frac{k_2k_{2z}-k_1k_{1z}}{2}e^{-2\ri d_1k_{1z}}+\frac{k_1k_{1z}+k_2k_{2z}}{2}\Big).
\end{split}
\end{equation*}
 Detailed derivation of the general formulas \eqref{generalreact}-\eqref{totaldensity} and corresponding reaction densities for Green's function in layered media can be found in \cite{wang2018}. {Although the density functions depend on the configuration of the layered media, they can be calculated at run time at any given point by	solving one $2\times2$ linear system and multiplying the results by some $2\times2$ matrices.}

\subsection{Multipole and local expansions for general reaction component} Consider the reaction field in the middle layers, i.e.,
\begin{equation}
u_{\ell\ell'}^{\rm r}(\bs r, \bs r')=u_{\ell\ell'}^{\uparrow}(\bs r, \bs r')+u_{\ell\ell'}^{\downarrow}(\bs r, \bs r').
\end{equation}
Define
\begin{equation}\label{mekernel}
\begin{split}
&\mathcal E_{\ell\ell'}^{\uparrow\uparrow}(\bs r, \bs r'):=e^{\ri\bs k_{\alpha}\cdot(\bs\rho-\bs\rho')+\ri k_{\ell z} (z-d_{\ell})+\ri k_{\ell' z}(z'-d_{\ell'})},\\
&\mathcal E_{\ell\ell'}^{\uparrow\downarrow}(\bs r, \bs r'):=e^{\ri\bs k_{\alpha}\cdot(\bs\rho-\bs\rho')+\ri k_{\ell z} (z-d_{\ell})+\ri k_{\ell' z}(d_{\ell'-1}-z')},\\
&\mathcal E_{\ell\ell'}^{\downarrow\uparrow}(\bs r, \bs r'):=e^{\ri\bs k_{\alpha}\cdot(\bs\rho-\bs\rho')+\ri k_{\ell z} (d_{\ell-1}-z)+\ri k_{\ell' z}(z'-d_{\ell'})},\\
&\mathcal E_{\ell\ell'}^{\downarrow\downarrow}(\bs r, \bs r'):=e^{\ri\bs k_{\alpha}\cdot(\bs\rho-\bs\rho')+\ri k_{\ell z} (d_{\ell-1}-z)+\ri k_{\ell' z}(d_{\ell'-1}-z')}.
\end{split}
\end{equation}
By expressions in \eqref{greenfuncomponent} and \eqref{totaldensity} we have a further decomposition
\begin{equation}
u_{\ell\ell'}^{\uparrow}(\bs r, \bs r')=u_{\ell\ell'}^{\uparrow\uparrow}(\bs r, \bs r')+u_{\ell\ell'}^{\uparrow\downarrow}(\bs r, \bs r'),\quad u_{\ell\ell'}^{\downarrow}(\bs r, \bs r')=u_{\ell\ell'}^{\downarrow\uparrow}(\bs r, \bs r')+u_{\ell\ell'}^{\downarrow\downarrow}(\bs r, \bs r'),
\end{equation}
and each component has a Sommerfeld-type integral representation:
\begin{equation}\label{generalcomponents}
u_{\ell\ell'}^{**}(\bs r, \bs r')=\frac{\ri}{8\pi^2 }\int_0^{\infty}\int_0^{2\pi}k_{\rho}\frac{\mathcal E_{\ell\ell'}^{**}(\bs r, \bs r')\sigma_{\ell\ell'}^{**}(k_{\rho})}{k_{\ell z}}d\alpha dk_{\rho}.
\end{equation}
Here and in the rest of this paper, $*$ stands for any one of arrows $\uparrow$, $\downarrow$, e.g., $u_{\ell\ell'}^{**}$ can be any of the four reaction components $u_{\ell\ell'}^{\uparrow\uparrow}, u_{\ell\ell'}^{\uparrow\downarrow}, u_{\ell\ell'}^{\downarrow\uparrow}, u_{\ell\ell'}^{\downarrow\downarrow}$. Note that the source and target coordinates are only involved in the exponential functions $\mathcal E_{\ell\ell'}^{**}(\bs r, \bs r')$. It is easy to make the following source-target separations
\begin{equation}\label{sourcetargetseparationsc}
\begin{split}
\mathcal E_{\ell\ell'}^{*\uparrow}(\bs r, \bs r')&=\mathcal E_{\ell\ell'}^{*\uparrow}(\bs r, \bs r_c^s)e^{\ri\bs k_{\alpha}\cdot(\bs\rho_c^s-\bs\rho')+\ri k_{\ell' z}(z'-z_c^s)},\\
\mathcal E_{\ell\ell'}^{*\downarrow}(\bs r, \bs r')&=\mathcal E_{\ell\ell'}^{*\downarrow}(\bs r, \bs r_c^s)e^{\ri\bs k_{\alpha}\cdot(\bs\rho_c^s-\bs\rho')-\ri k_{\ell' z}(z'-z_c^s)}
\end{split}
\end{equation}
by inserting the source center $\bs r_c^s$, and
\begin{equation}\label{sourcetargetseparationtc}
\begin{split}
\mathcal E_{\ell\ell'}^{\uparrow*}(\bs r, \bs r')&=\mathcal E_{\ell\ell'}^{\uparrow*}(\bs r_c^t, \bs r')e^{\ri\bs k_{\alpha}\cdot(\bs\rho-\bs\rho_c^t)+\ri k_{\ell z}(z-z_c^t)},\\
\mathcal E_{\ell\ell'}^{\downarrow*}(\bs r, \bs r')&=\mathcal E_{\ell\ell'}^{\downarrow*}(\bs r_c^t, \bs r')e^{\ri\bs k_{\alpha}\cdot(\bs\rho-\bs\rho_c^t)-\ri k_{\ell z}(z-z_c^t)}
\end{split}
\end{equation}
by inserting the target center $\bs r_c^t$. Here $\bs\rho_c^s=(x_c^s, y_c^s)$, $\bs\rho_c^t=(x_c^t, y_c^t)$ are the $x, y$-coordinates of the centers $\bs r_c^s$ and $\bs r_c^t$.
Moreover, Funk-Hecke formula \eqref{extfunkhecke} gives spherical harmonic expansions for the following plane waves:
\begin{equation*}
\begin{split}
e^{\ri\bs k_{\alpha}\cdot(\bs\rho_c^s-\bs\rho')\pm\ri k_{\ell' z}(z'-z_c^s)}&=4\pi\sum\limits_{n=0}^{\infty}\sum\limits_{m=-n}^{n} (-\ri)^n(\mp 1)^{n+m}j_n(k_{\ell'}r_s')\overline{Y_n^{m}(\theta_s',\varphi_s')}\widehat{P}_n^m\Big(\frac{k_{\ell'z}}{k_{\ell'}}\Big)e^{\ri m\alpha},\\
e^{\ri\bs k_{\alpha}\cdot(\bs\rho-\bs\rho_c^t)\pm\ri k_{\ell z}(z-z_c^t)}&=4\pi\sum\limits_{n=0}^{\infty}\sum\limits_{m=-n}^{n} \ri^n(\pm 1)^{n+m}j_n(k_{\ell}r_t)Y_n^{m}(\theta_t,\varphi_t)\widehat{P}_n^m\Big(\frac{k_{\ell z}}{k_{\ell}}\Big)e^{-\ri m\alpha},
\end{split}
\end{equation*}
where the fact that $Y_n^{m}(\pi-\theta,\varphi)=(-1)^{n+m}Y_n^{m}(\theta,\varphi)$ is used. Applying these expansions and source-target separation \eqref{sourcetargetseparationsc} and \eqref{sourcetargetseparationtc}, we obtain
\begin{equation*}
\begin{split}
\mathcal E_{\ell\ell'}^{*\uparrow}(\bs r, \bs r')&=\mathcal E_{\ell\ell'}^{*\uparrow}(\bs r, \bs r_c^s)\sum\limits_{n=0}^{\infty}\sum\limits_{m=-n}^{n}4\pi j_n(k_{\ell'}r_s')\overline{Y_n^{m}(\theta_s',\varphi_s')} (- 1)^{m}\ri^n\widehat{P}_n^m\Big(\frac{k_{\ell'z}}{k_{\ell'}}\Big)e^{\ri m\alpha},\\
\mathcal E_{\ell\ell'}^{*\downarrow}(\bs r, \bs r')&=\mathcal E_{\ell\ell'}^{*\downarrow}(\bs r, \bs r_c^s)\sum\limits_{n=0}^{\infty}\sum\limits_{m=-n}^{n}4\pi j_n(k_{\ell'}r_s')\overline{Y_n^{m}(\theta_s',\varphi_s')}(-1)^n\ri^n\widehat{P}_n^m\Big(\frac{k_{\ell'z}}{k_{\ell'}}\Big)e^{\ri m\alpha},
\end{split}
\end{equation*}
and
\begin{equation*}
\begin{split}
\mathcal E_{\ell\ell'}^{\uparrow*}(\bs r, \bs r')&=\mathcal E_{\ell\ell'}^{\uparrow*}(\bs r_c^t, \bs r')4\pi\sum\limits_{n=0}^{\infty}\sum\limits_{m=-n}^{n} j_n(k_{\ell}r_t)Y_n^{m}(\theta_t,\varphi_t)\ri^n\widehat{P}_n^m\Big(\frac{k_{\ell z}}{k_{\ell}}\Big)e^{-\ri m\alpha},\\
\mathcal E_{\ell\ell'}^{\downarrow*}(\bs r, \bs r')&=\mathcal E_{\ell\ell'}^{\downarrow*}(\bs r_c^t, \bs r')4\pi\sum\limits_{n=0}^{\infty}\sum\limits_{m=-n}^{n} j_n(k_{\ell}r_t)Y_n^{m}(\theta_t,\varphi_t)(- 1)^{n+m}\ri^n\widehat{P}_n^m\Big(\frac{k_{\ell z}}{k_{\ell}}\Big)e^{-\ri m\alpha}.
\end{split}
\end{equation*}
Substituting the above identities into \eqref{generalcomponents}, we obtain the following multipole expansion:
\begin{equation}\label{melayerupgoing}
u_{\ell\ell'}^{**}(\bs r, \bs r')=\sum\limits_{n=0}^{\infty}\sum\limits_{m=-n}^{n}  M_{nm}\mathcal F_{nm}^{**}(\bs r, \bs r_c^s),\quad M_{nm}=4\pi j_n(k_{\ell'}r_s')\overline{Y_n^{m}(\theta_s',\varphi_s')},
\end{equation}
with the expansion functions given by Sommerfeld-type integrals
\begin{equation}\label{expfuncreaction}
\begin{split}
\mathcal F_{nm}^{*\uparrow}(\bs r, \bs r_c^s)=&\frac{\ri}{8\pi^2}\int_0^{\infty}\int_0^{2\pi}k_{\rho}\frac{(-1)^{m}\mathcal E_{\ell\ell'}^{*\uparrow}(\bs r, \bs r_c^s)\sigma_{\ell\ell'}^{*\uparrow}(k_{\rho})}{k_{\ell z}}\ri^{n}\widehat{P}_n^m\Big(\frac{k_{\ell'z}}{k_{\ell'}}\Big)e^{\ri m\alpha}d\alpha dk_{\rho},\\
\mathcal F_{nm}^{*\downarrow}(\bs r, \bs r_c^s)=&\frac{\ri}{8\pi^2}\int_0^{\infty}\int_0^{2\pi}k_{\rho}\frac{(-1)^n\mathcal E_{\ell\ell'}^{*\downarrow}(\bs r, \bs r_c^s)\sigma_{\ell\ell'}^{*\downarrow}(k_{\rho})}{k_{\ell z}}\ri^n\widehat{P}_n^m\Big(\frac{k_{\ell'z}}{k_{\ell'}}\Big)e^{\ri m\alpha}d\alpha dk_{\rho}.
\end{split}
\end{equation}
It is worthy to point out that $\mathcal F_{nm}^{**}(\bs r, \bs r_c^s)$ is only defined for $\bs r$ in the $\ell$-th layer. Therefore, it could be seen as singular function for $\bs r$ outside the $\ell$-th layer as reaction field produced by polarization charges there (more discussion on this issue will be given below). That is why we keep using the ``multipole expansion" for expansion \eqref{melayerupgoing}.
Similarly, we obtain local expansion
\begin{equation}\label{lelayerupgoing}
u_{\ell\ell'}^{**}(\bs r, \bs r')=\sum\limits_{n=0}^{\infty}\sum\limits_{m=-n}^{n} L_{nm}^{**}j_n(k_{\ell}r_t)Y_n^m(\theta_t,\varphi_t),
\end{equation}
with coefficients given by
\begin{equation}\label{lecoeffup}
\begin{split}
L_{nm}^{\uparrow*}=&\frac{\ri}{2\pi }\int_0^{\infty}\int_0^{2\pi}k_{\rho}\frac{\mathcal E_{\ell\ell^{\prime}}^{\uparrow*}(\bs r_c^t, \bs r')\sigma_{\ell\ell^{\prime}}^{\uparrow*}(k_{\rho})}{k_{\ell z}}\ri^n\widehat{P}_n^m\Big(\frac{k_{\ell z}}{k_{\ell}}\Big)e^{-\ri m\alpha}d\alpha dk_{\rho},\\
L_{nm}^{\downarrow*}=&\frac{(-1)^{n+m}\ri}{2\pi }\int_0^{\infty}\int_0^{2\pi}k_{\rho}\frac{\mathcal E_{\ell\ell^{\prime}}^{\downarrow*}(\bs r_c^t, \bs r')\sigma_{\ell\ell^{\prime}}^{\downarrow*}(k_{\rho})}{k_{\ell z}}\ri^n\widehat{P}_n^m\Big(\frac{k_{\ell z}}{k_{\ell}}\Big)e^{-\ri m\alpha}d\alpha dk_{\rho}.
\end{split}
\end{equation}
According to the definition of $\mathcal E_{\ell\ell'}^{**}(\bs r, \bs r_c^s)$ and $\mathcal E_{\ell\ell'}^{**}(\bs r_c^t, \bs r')$ in \eqref{mekernel},  the centers $\bs r_c^s$ and $\bs r_c^t$ have to satisfy
\begin{equation}\label{centercond}
\begin{split}
z_c^s>d_{\ell'}, \; z_c^t>d_{\ell},\quad {\rm for \;}u^{\uparrow\uparrow}_{\ell\ell'}(\bs r, \bs r');\quad z_c^s<d_{\ell'-1}, \; z_c^t<d_{\ell-1},\quad {\rm for \;}u^{\downarrow\downarrow}_{\ell\ell'}(\bs r, \bs r');\\
z_c^s<d_{\ell'-1}, \; z_c^t>d_{\ell},\quad {\rm for \;}u^{\uparrow\downarrow}_{\ell\ell'}(\bs r, \bs r');\quad z_c^s>d_{\ell'}, \; z_c^t<d_{\ell-1},\quad {\rm for \;}u^{\downarrow\uparrow}_{\ell\ell'}(\bs r, \bs r')
\end{split}
\end{equation}
to ensure the exponential decay in $\mathcal E_{\ell\ell'}^{**}(\bs r, \bs r_c^s)$ and $\mathcal E_{\ell\ell'}^{**}(\bs r_c^t, \bs r')$ as $k_{\rho}\rightarrow\infty$ and hence the convergence of the corresponding Sommerfeld-type integrals in \eqref{expfuncreaction} and \eqref{lecoeffup}. These restriction can be met in practice, since we are considering targets in the $\ell$-th layer and sources in the $\ell'$-th layer.

\subsection{Reaction components, associated equivalent polarization sources, and multipole and local expansions}
It is well known that the $h$-expansion \eqref{freespace3dmulexp} and the $j$-expansion \eqref{freespace3dlocexp} for the free space Green's function have convergence rates of order $\mathcal O\Big(\Big(\frac{|\bs r'-\bs r_c^s|}{|\bs r-\bs r_c^s|}\Big)^p\Big)$ and $\mathcal O\Big(\Big(\frac{|\bs r-\bs r_c^t|}{|\bs r_c^t-\bs r'|}\Big)^p\Big)$, respectively. That is both multipole and local expansion converge exponentially as the target moving away from the source. These convergence results are the key for the success of the hierarchical  tree structure design in FMM. However, multipole and local expansions \eqref{melayerupgoing}-\eqref{lecoeffup} for reaction components have a different convergence behavior. According to the convergence analysis for multipole and local expansions of 2-D Green's function of Helmholtz equation in layered media (cf. \cite{zhang2018exponential}), we expect convergence estimates for the ME at $\bs r_c^s$ and the LE at $\bs r_c^t$ for the reaction components
\begin{equation}\label{melayerapp}
\begin{split}
\Big|u_{\ell\ell'}^{**}(\bs r, \bs r')-\sum\limits_{n=0}^{p}\sum\limits_{m=-n}^{n}  M_{nm}\mathcal F_{nm}^{**}(\bs r, \bs r_c^s)\Big|=\mathcal{O}\left(\left(\frac{|\bs r'-\bs r_c^s|}{d^{**}(\bs r, \bs r_c^s)}\right)^p\right),\\
\Big|u_{\ell\ell'}^{**}(\bs r, \bs r')-\sum\limits_{n=0}^{p}\sum\limits_{m=-n}^{n}  L_{nm}^{**}j_n(k_{\ell}r_t)Y_n^m(\theta_t,\varphi_t)\Big|=\mathcal{O}\left(\left(\frac{|\bs r-\bs r_c^t|}{d^{**}(\bs r_c^t, \bs r')}\right)^p\right),
\end{split}
\end{equation}
respectively, where
\begin{equation}\label{redefinedist}
\begin{split}
d^{\uparrow\uparrow}(\bs r, \bs r_c^s)=&\sqrt{(x-x_c^s)^2+(y-y_c^s)^2+(z-d_{\ell}+z_c^s-d_{\ell'})^2},\\
d^{\uparrow\downarrow}(\bs r, \bs r_c^s)=&\sqrt{(x-x_c^s)^2+(y-y_c^s)^2+(z-d_{\ell}+d_{\ell'-1}-z_c^s)^2},\\
d^{\downarrow\uparrow}(\bs r, \bs r_c^s)=&\sqrt{(x-x_c^s)^2+(y-y_c^s)^2+(d_{\ell-1}-z+z_c^s-d_{\ell'})^2},\\
d^{\downarrow\downarrow}(\bs r, \bs r_c^s)=&\sqrt{(x-x_c^s)^2+(y-y_c^s)^2+(d_{\ell-1}-z+d_{\ell'-1}-z_c^s)^2}.
\end{split}
\end{equation}
Suppose $\bs r_c^s$ is a center close to source $\bs r'$ with a fixed distance $|\bs r'-\bs r_c^s|$, \eqref{melayerapp} indicates an important fact that the error of the truncated multipole expansion is not determined by the Euclidean distance between source center $\bs r_c^s$ and target $\bs r$ as in the free space case. Actually, the distances along $z$-direction have been replaced by summations of the distances between $\bs r$, $\bs r_c^s$ and corresponding nearest interfaces of the layered media. Similar conclusion can also be obtained for local expansion \eqref{lelayerupgoing}.

\begin{rem}
	There are two special cases, i.e., $d^{\uparrow\downarrow}(\bs r,\bs r_c^s)=|\bs r-\bs r_c^s|$ if $\bs r$ and $\bs r_c^s$ are in the $\ell$-th and $(\ell+1)$-th layer and $d^{\downarrow\uparrow}(\bs r,\bs r_c^s)=|\bs r-\bs r_c^s|$ if $\bs r$ and $\bs r_c^s$ are in the $\ell$-th and $(\ell-1)$-th layer. Therefore, multipole and local expansions of $u_{\ell\ell+1}^{\uparrow\downarrow}(\bs r, \bs r_c^s)$ and $u_{\ell\ell-1}^{\downarrow\uparrow}(\bs r, \bs r_c^s)$ have the same convergence behavior as that of free space Green's function.
\end{rem}

We will give some numerical examples to show the convergence behavior of the multipole expansions in \eqref{melayerapp}. Consider the multipole expansions of $u_{11}^{\uparrow\uparrow}$ and $u_{11}^{\uparrow\downarrow}$ in a three-layer media with $k_{0}=0.8$, $k_{1}=1.5$, $k_{2}=2.0$, $d_0=0, d_1=-2.0$. In all the following examples, we fix $x=x'=0.625$, $y=y'=0.5$, $x_c^s=1.0,\; y_c^s=0.5$ for the coordinates of the target $\bs r$, source $\bs r'$ and source center $\bs r_c^s$, respectively. Moreover, we will have the distance $|\bs r'-\bs r_c^s|=0.375$ fixed by keeping $z'=z_c^s$.  For both components, we shall test three groups of $z$-coordinates given as follows
\begin{equation*}
\begin{split}
&u_{11}^{\uparrow\uparrow}: z=z'=z_c^s=-1.8; \quad z=-1.8,\; z'=z_c^s=-1.0;\quad z=z'=z_c^s=-0.5;\\
&u_{11}^{\uparrow\downarrow}: z=-1.8,\;z'=z_c^s=-0.2; \; z=-1.8,\; z'=z_c^s=-1.0;\; z=-0.5,\;z'=z_c^s=-1.5.
\end{split}
\end{equation*}
The relative errors against truncation number $p$ are depicted in Fig. \ref{meconvergence}. { We also use a linear polynomial to fit the value $a$ in the expected exponential convergence $\mathcal{O}(a^p)$ via
$$\log err=p\log a+C.$$
The fitted values of $a$ for all cases in Fig. \ref{meconvergence} are given in the captions of the sub-figures.} The results clearly show that multipole expansion of $u_{11}^{**}$ converges faster as the distance $d^{**}(\bs r, \bs r')$ rather than Euclidean distance $|\bs r-\bs r'|$ increases. We also see that multipole expansion of $u_{11}^{\uparrow\downarrow}$ converges much faster than that of $u_{11}^{\uparrow\uparrow}$ as $d^{\uparrow\downarrow}(\bs r, \bs r_c^s)=d^{\uparrow\uparrow}(\bs r, \bs r_c^s)$. This can be explained by the extra exponential decay term inside the density $\sigma_{11}^{\uparrow\downarrow}(k_{\rho})$, see Fig. \ref{deformedcontour}.
\begin{figure}[ht!]
	\center
	\subfigure[$u_{11}^{\uparrow\uparrow}(\bs r, \bs r')$ and $a=0.63,    0.25,   0.10$]{\includegraphics[scale=0.35]{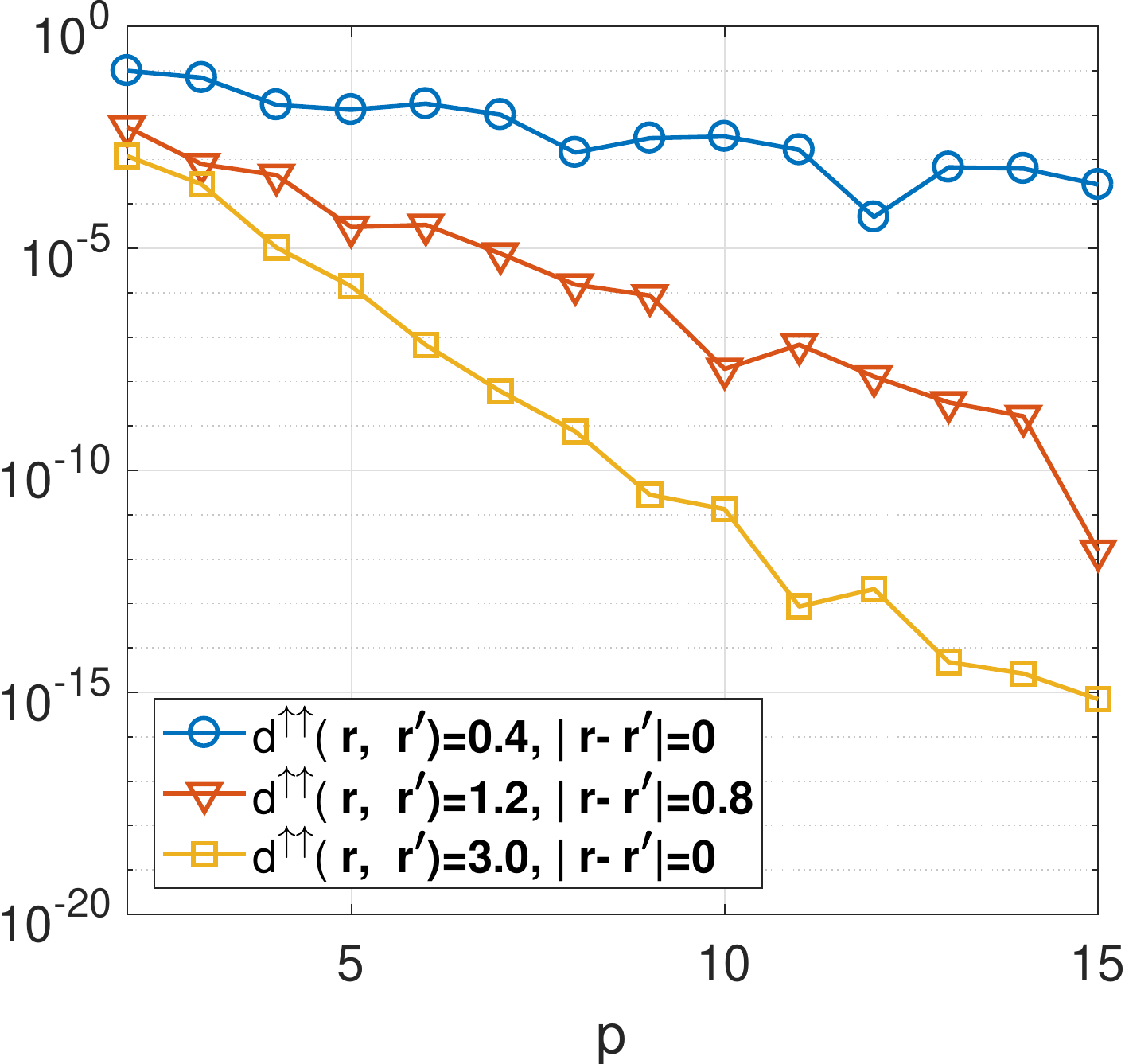}} \qquad
	\subfigure[$u_{11}^{\uparrow\downarrow}(\bs r, \bs r')$ and $a=0.13, 0.10, 0.10$]{\includegraphics[scale=0.35]{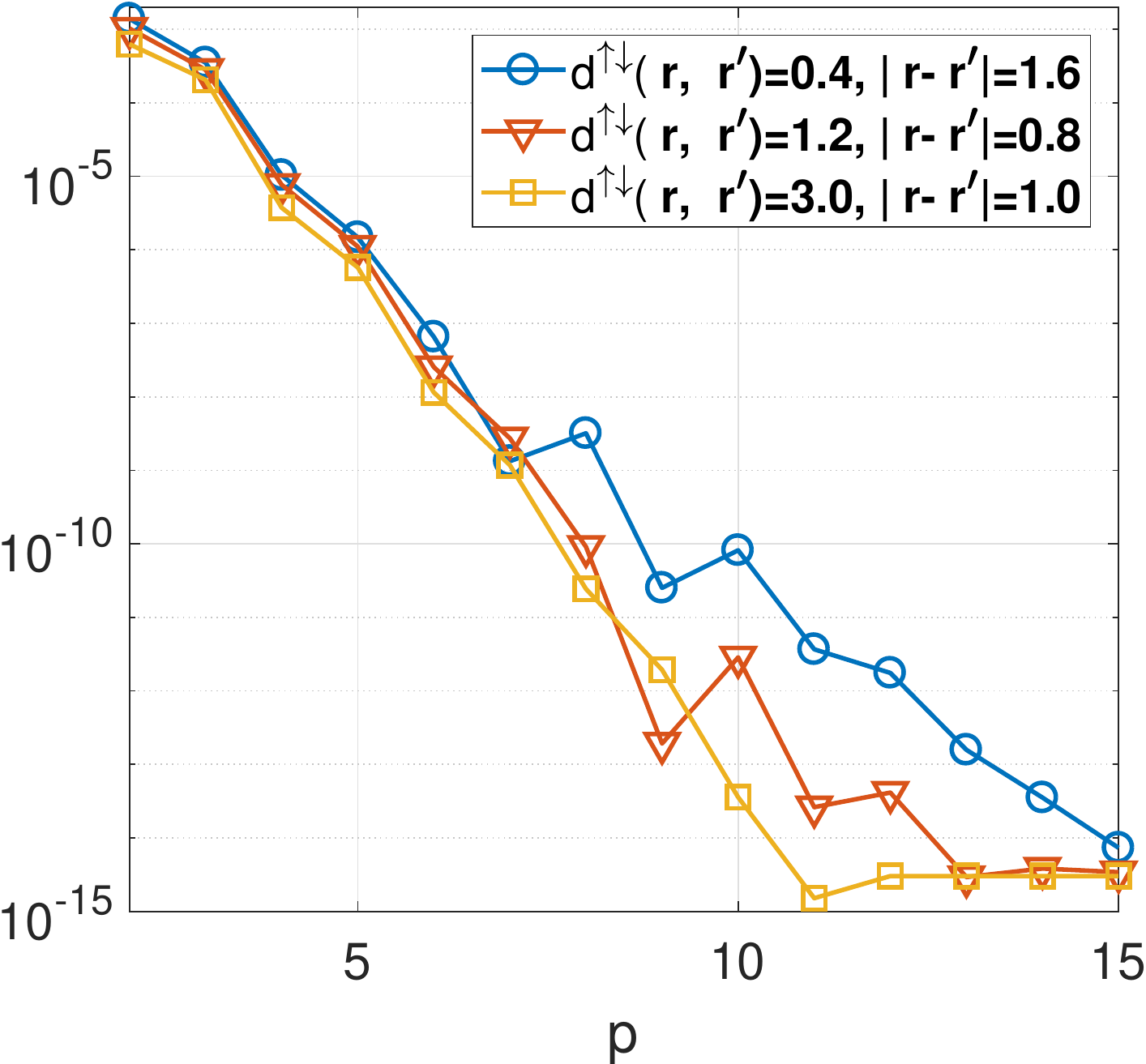}}
	\caption{Errors of the multipole expansions vs. truncation order $p$.}%
	\label{meconvergence}%
\end{figure}
\begin{figure}[ptbh]
	\center
	\subfigure[$u_{\ell\ell'}^{\uparrow\uparrow}$]{\includegraphics[scale=0.4]{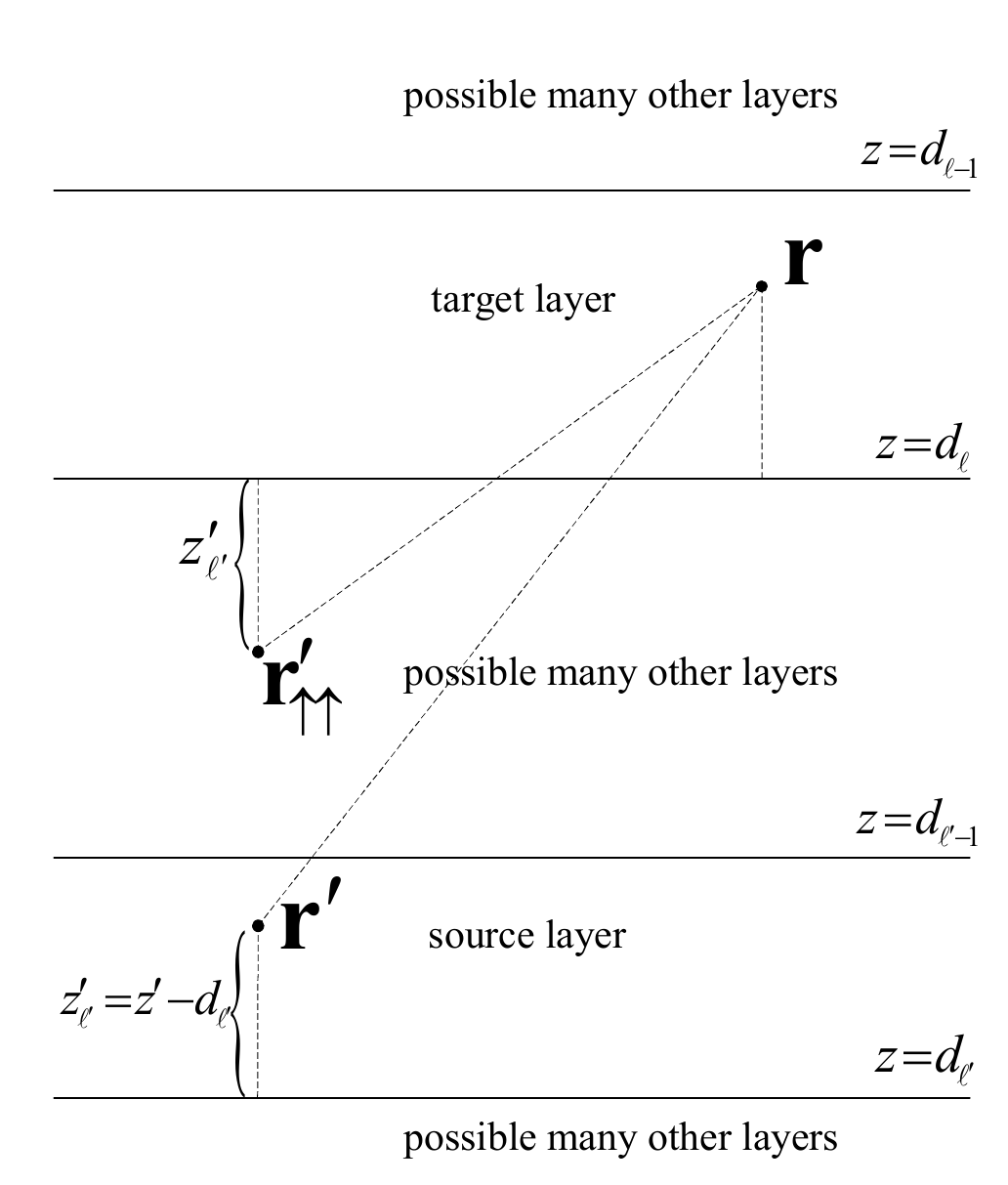}} \qquad
	\subfigure[$u_{\ell\ell'}^{\uparrow\downarrow}$]{\includegraphics[scale=0.4]{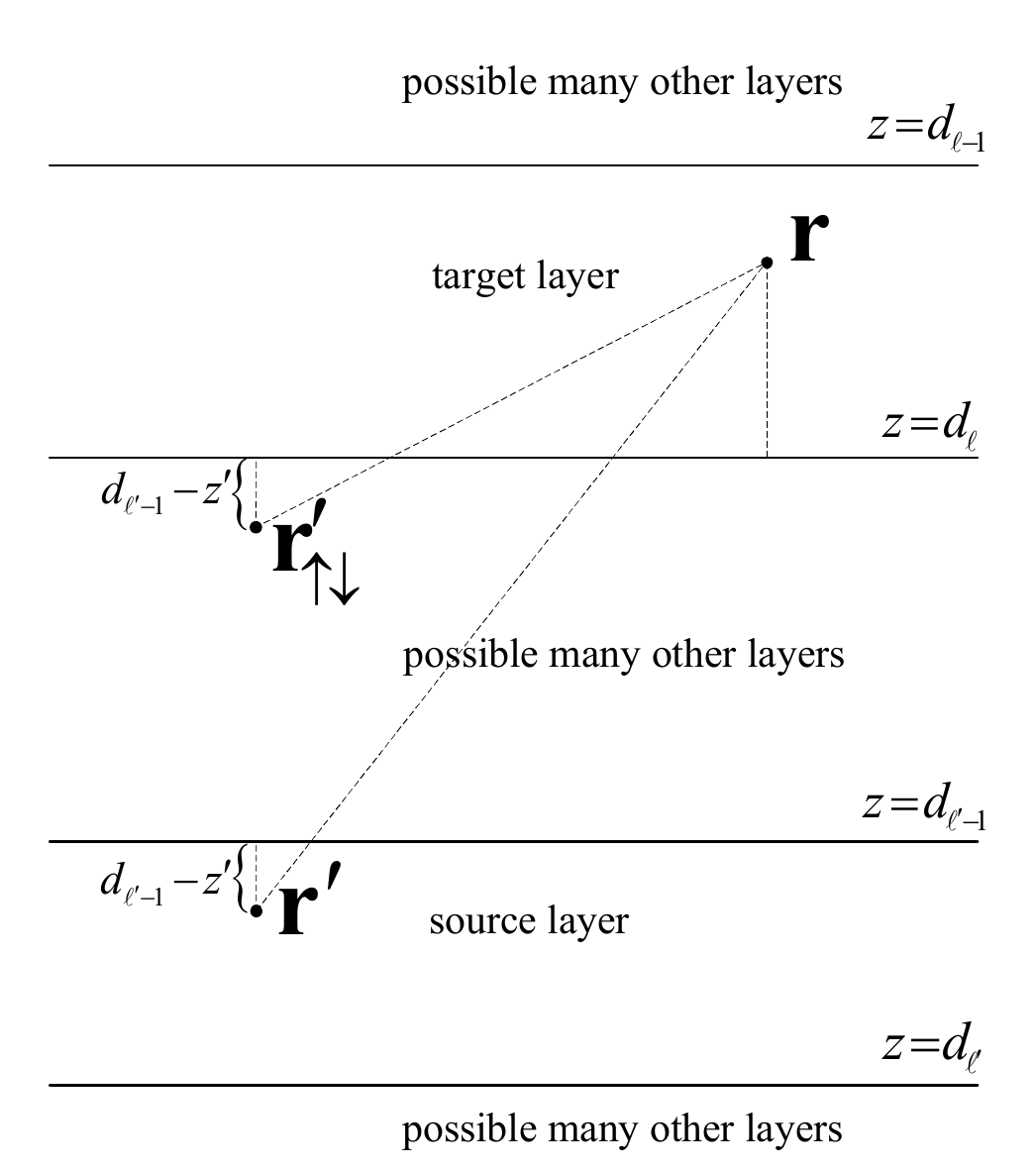}} \\

	\subfigure[$u_{\ell\ell'}^{\downarrow\uparrow}$]{\includegraphics[scale=0.4]{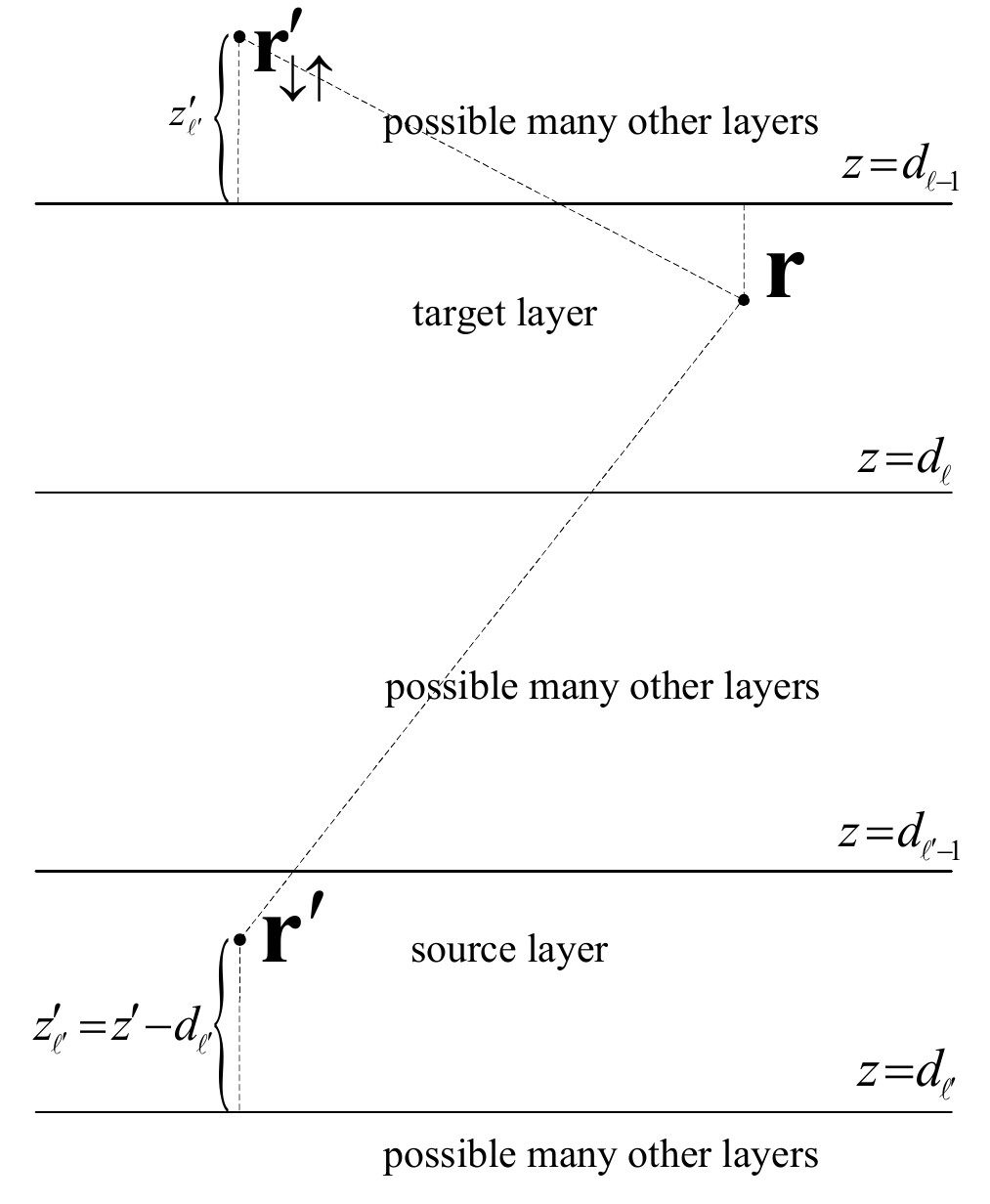}} \qquad
	\subfigure[$u_{\ell\ell'}^{\downarrow\downarrow}$]{\includegraphics[scale=0.4]{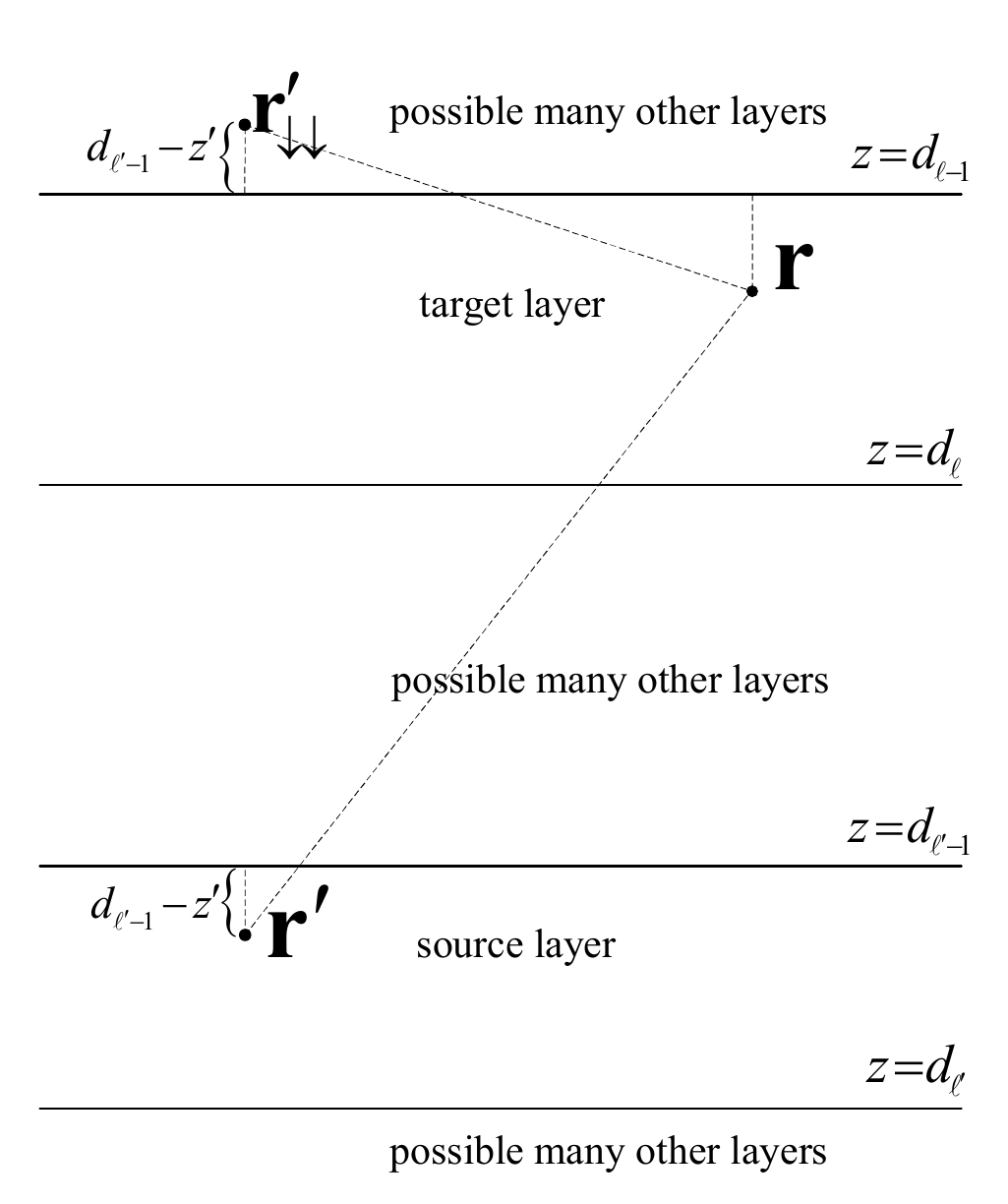}}
	\caption{Location of equivalent polarization sources for the computation of $u_{\ell\ell'}^{**}$.}%
	\label{sourceimages}%
\end{figure}

\noindent {\bf Polarization source for reaction field.} In extending the FMM to the reaction field component of the layered Green's function, the hierarchical tree structure design as in free space FMM relies on using the Euclidean distance between source and target to determine either direct computation or multipole and local expansions are used for the computation of source-target interactions. Therefore, multipole and local expansions as given in \eqref{melayerupgoing}-\eqref{lecoeffup} are generally not compatible with the hierarchical design of FMM. Considering the convergence behavior of ME in \eqref{melayerapp} and \eqref{redefinedist}, we introduce equivalent polarization sources for the four types of reaction fields (see. Fig. \ref{sourceimages})
\begin{equation}\label{eqpolarizedsource}
\begin{split}
&\bs r'_{\uparrow\uparrow}=(x', y', d_{\ell}-(z'-d_{\ell'})),\quad\quad \bs r'_{\uparrow\downarrow}=(x', y', d_{\ell}-(d_{\ell'-1}-z')),\\
&\bs r'_{\downarrow\uparrow}=(x', y', d_{\ell-1}+(z'-d_{\ell'})),\quad \bs r'_{\downarrow\downarrow}=(x', y', d_{\ell-1}+(d_{\ell'-1}-z'))
\end{split}
\end{equation}
{which satisfy
\begin{equation}\label{distancerelation}
d^{\uparrow*}(\bs r, \bs r'_{\uparrow*})=|\bs r-\bs r'_{\uparrow*}|,\quad  {\rm if}\; z>d_{\ell}; \quad d^{\downarrow*}(\bs r, \bs r'_{\downarrow*})=|\bs r-\bs r'_{\downarrow*}|,\quad  {\rm if}\; z<d_{\ell-1}.
\end{equation}}

Also, we define
\begin{equation}\label{zexponentialimage}
\begin{split}
&\widetilde{\mathcal{Z}}_{\ell\ell'}^{\uparrow}(z, z_s):=e^{\ri k_{\ell z} (z-d_{\ell})+\ri k_{\ell' z}(d_{\ell}-z_s)},\quad \widetilde{\mathcal{Z}}_{\ell\ell'}^{\downarrow}(z, z_s):=e^{\ri k_{\ell z} (d_{\ell-1}-z)+\ri k_{\ell' z}(z_s-d_{\ell-1})},
\end{split}
\end{equation}
and
\begin{equation}\label{mekernelimage}
\begin{split}
&\widetilde{\mathcal E}_{\ell\ell'}^{*}(\bs r, \bs r_s):=e^{\ri\bs k_{\alpha}\cdot(\bs\rho-\bs\rho_s)}\widetilde{\mathcal{Z}}_{\ell\ell'}^{*}(z, z_s).
\end{split}
\end{equation}
Recall the expressions \eqref{mekernel}, we can verify that
\begin{equation}
\begin{split}
\mathcal E_{\ell\ell'}^{\uparrow*}(\bs r, \bs r')=\widetilde{\mathcal E}_{\ell\ell'}^{\uparrow}(\bs r, \bs r'_{\uparrow*}),\quad
\mathcal E_{\ell\ell'}^{\downarrow*}(\bs r, \bs r')=\widetilde{\mathcal E}_{\ell\ell'}^{\downarrow}(\bs r, \bs r'_{\downarrow*}).
\end{split}
\end{equation}
Therefore, the reaction components can be re-expressed using equivalent polarization sources as follows
\begin{equation}\label{generalcomponentsimag}
\begin{split}
u_{\ell\ell'}^{\uparrow*}(\bs r, \bs r')=\tilde u_{\ell\ell'}^{\uparrow*}(\bs r, \bs r'_{\uparrow*}):=\frac{\ri}{8\pi^2 }\int_0^{\infty}\int_0^{2\pi}k_{\rho}\frac{\widetilde{\mathcal E}_{\ell\ell'}^{\uparrow}(\bs r, \bs r'_{\uparrow*})\sigma_{\ell\ell'}^{\uparrow*}(k_{\rho})}{k_{\ell z}}d\alpha dk_{\rho},\\
u_{\ell\ell'}^{\downarrow*}(\bs r, \bs r')=\tilde u_{\ell\ell'}^{\downarrow*}(\bs r, \bs r'_{\downarrow*}):=\frac{\ri}{8\pi^2 }\int_0^{\infty}\int_0^{2\pi}k_{\rho}\frac{\widetilde{\mathcal E}_{\ell\ell'}^{\downarrow}(\bs r, \bs r'_{\downarrow*})\sigma_{\ell\ell'}^{\downarrow*}(k_{\rho})}{k_{\ell z}}d\alpha dk_{\rho}.
\end{split}
\end{equation}

Note that $\tilde u_{\ell\ell'}^{\uparrow*}$ and $\tilde u_{\ell\ell'}^{\downarrow*}$ have the same form as the two special cases $u_{\ell\ell+1}^{\uparrow\downarrow}$ and $u_{\ell\ell-1}^{\downarrow\uparrow}$, respectively, except for wave number $k_{\ell'z}$ and densities $\sigma_{\ell\ell'}^{**}$ in the $\ell'$-th layer. Following the same derivation before,
we obtain the following multipole expansion
\begin{equation}\label{melayerupgoingimage}
\begin{split}
\tilde u_{\ell\ell'}^{**}(\bs r, \bs r'_{**})=\sum\limits_{n=0}^{\infty}\sum\limits_{m=-n}^{n}  M_{nm}^{**}\widetilde{\mathcal F}_{nm}^{**}(\bs r, \bs r_c^{**}), \quad M_{nm}^{**}=4\pi j_n(k_{\ell'}r_s^{**})\overline{Y_n^{m}(\theta_s^{**},\varphi_s^{**})},
\end{split}
\end{equation}
at equivalent polarization source centers $\bs r_c^{**}$ and local expansion
\begin{equation}\label{lelayerimage}
\begin{split}
\tilde u_{\ell\ell'}^{**}(\bs r, \bs r'_{**})=\sum\limits_{n=0}^{\infty}\sum\limits_{m=-n}^{n} L_{nm}^{**}j_n(k_{\ell}r_t)Y_n^m(\theta_t,\varphi_t)
\end{split}
\end{equation}
at target center $\bs r_c^t$. Here, $(r_s^{**}, \theta_s^{**}, \varphi_s^{**})$ are spherical coordinates of $\bs r'_{**}-\bs r_c^{**}$, $\widetilde{\mathcal F}_{nm}^{**}(\bs r, \bs r_c^{**})$ are represented by Sommerfeld-type integrals
\begin{equation}\label{mebasis}
\begin{split}
\widetilde{\mathcal F}_{nm}^{\uparrow*}(\bs r, \bs r_c^{\uparrow*})=&\frac{\ri}{8\pi^2}\int_0^{\infty}\int_0^{2\pi}k_{\rho}\frac{\widetilde{\mathcal E}_{\ell\ell'}^{\uparrow}(\bs r, \bs r_c^{\uparrow*})\sigma_{\ell\ell'}^{\uparrow*}(k_{\rho})}{k_{\ell z}}(-\ri)^n\widehat{P}_n^m\Big(\frac{k_{\ell'z}}{k_{\ell'}}\Big)e^{\ri m\alpha}d\alpha dk_{\rho},\\
\widetilde{\mathcal F}_{nm}^{\downarrow*}(\bs r, \bs r_c^{\downarrow*})=&\frac{\ri}{8\pi^2}\int_0^{\infty}\int_0^{2\pi}k_{\rho}\frac{\widetilde{\mathcal E}_{\ell\ell'}^{\downarrow}(\bs r, \bs r_c^{\downarrow*})\sigma_{\ell\ell'}^{\downarrow*}(k_{\rho})}{k_{\ell z}}(-\ri)^n\widehat{P}_n^m\Big(\frac{k_{\ell'z}}{k_{\ell'}}\Big)e^{\ri m\alpha}d\alpha dk_{\rho},
\end{split}
\end{equation}
and the local expansion coefficients are given by
\begin{equation}\label{lecoeffimage}
\begin{split}
L_{nm}^{\uparrow*}=&\frac{\ri}{2\pi }\int_0^{\infty}\int_0^{2\pi}k_{\rho}\frac{\widetilde{\mathcal E}_{\ell\ell^{\prime}}^{\uparrow}(\bs r_c^t, \bs r'_{\uparrow*})\sigma_{\ell\ell^{\prime}}^{\uparrow*}(k_{\rho})}{k_{\ell z}}\ri^n\widehat{P}_n^m\Big(\frac{k_{\ell z}}{k_{\ell}}\Big)e^{-\ri m\alpha}d\alpha dk_{\rho},\\
L_{nm}^{\downarrow*}=&\frac{\ri}{2\pi }\int_0^{\infty}\int_0^{2\pi}k_{\rho}\frac{(-1)^{n+m}\widetilde{\mathcal E}_{\ell\ell^{\prime}}^{\downarrow}(\bs r_c^t, \bs r'_{\downarrow*})\sigma_{\ell\ell^{\prime}}^{\uparrow*}(k_{\rho})}{k_{\ell z}}\ri^n\widehat{P}_n^m\Big(\frac{k_{\ell z}}{k_{\ell}}\Big)e^{-\ri m\alpha}d\alpha dk_{\rho}.
\end{split}
\end{equation}
{Similar to the restrictions in \eqref{centercond}, the centers $r_c^{**}$ and $\bs r_c^t$ are required to satisfy
\begin{equation}\label{imagecentercond}
z_c^{\uparrow*}<d_{\ell}, \quad z_c^{\downarrow*}>d_{\ell-1},\quad z_c^t>d_{\ell} \;\; {\rm for}\;\; \tilde u_{\ell\ell'}^{\uparrow*}(\bs r, \bs r'_{\uparrow*}); \quad z_c^t<d_{\ell-1} \;\; {\rm for}\;\; \tilde u_{\ell\ell'}^{\downarrow*}(\bs r, \bs r'_{\downarrow*}).
\end{equation}}
Recall convergence results \eqref{melayerapp} and the fact \eqref{distancerelation},
we conclude that ME \eqref{melayerupgoingimage} now satisfies
\begin{equation}
\begin{split}
\Big|\tilde u_{\ell\ell'}^{**}(\bs r, \bs r'_{**})-\sum\limits_{n=0}^{p}\sum\limits_{m=-n}^{n}  M_{nm}^{**}\widetilde{\mathcal F}_{nm}^{**}(\bs r, \bs r_c^{**})\Big|=\mathcal{O}\left(\left(\frac{|\bs r'_{**}-\bs r_c^{**}|}{|\bs r-\bs r_c^{**}|}\right)^p\right).
\end{split}
\end{equation}
As a result, the Euclidean distance  $|\bs r-\bs r'_{**}|$ can be used to determine if multipole expansions \eqref{melayerupgoingimage} are good approximations to the far reaction field. Similarly, the local expansion \eqref{lelayerimage} for $\tilde u_{\ell\ell'}^{**}(\bs r, \bs r'_{**})$ have a convergence of order $\mathcal{O}\left(\left(\frac{|\bs r-\bs r_c^t|}{|\bs r_c^t-\bs r'_{**}|}\right)^p\right)$. These convergence results ensure that the hierarchical design can be applied in FMM with kernels $\tilde u_{\ell\ell'}^{**}(\bs r, \bs r'_{**})$, ME \eqref{melayerupgoingimage} and LE \eqref{lelayerimage}.

Next, we discuss the center shifting and translation for ME \eqref{melayerupgoingimage} and LE \eqref{lelayerimage}. A desirable feature of the expansions of reaction components discussed above is that the formula \eqref{melayerupgoingimage} for the ME coefficients and the formula \eqref{lelayerimage} for the LE have exactly the same form as the formulas of $h$-expansion coefficients and $j$-expansion for free space Green's function. Therefore, we can see that center shifting for multipole and local expansions are exactly the same as free space case given in \eqref{metome}.

We only need to derive the translation operator from ME \eqref{melayerupgoingimage} to LE \eqref{lelayerimage}. Recall the definition of exponential functions in \eqref{mekernelimage}, $\widetilde{\mathcal E}_{\ell\ell'}^{\uparrow}(\bs r, \bs r_c^{\uparrow*})$ and $\widetilde{\mathcal E}_{\ell\ell'}^{\downarrow}(\bs r, \bs r_c^{\downarrow*})$ have the following splitting
\begin{equation*}
\begin{split}
\widetilde{\mathcal E}_{\ell\ell'}^{\uparrow}(\bs r, \bs r_c^{\uparrow*})&=\widetilde{\mathcal E}_{\ell\ell'}^{\uparrow}(\bs r_c^t, \bs r_c^{\uparrow*})e^{\ri\bs k_{\alpha}\cdot(\bs\rho-\bs\rho_c^t)}e^{\ri k_{\ell z} (z-z_c^t)},\\
\widetilde{\mathcal E}_{\ell\ell'}^{\downarrow}(\bs r, \bs r_c^{\downarrow*})&=\widetilde{\mathcal E}_{\ell\ell'}^{\downarrow}(\bs r_c^t, \bs r_c^{\downarrow*})e^{\ri\bs k_{\alpha}\cdot(\bs\rho-\bs\rho_c^t)}e^{-\ri k_{\ell z} (z-z_c^t)}.
\end{split}
\end{equation*}
Applying Funk-Hecke formula \eqref{extfunkhecke} again, we obtain
\begin{equation*}
e^{\ri\bs k_{\alpha}\cdot(\bs\rho-\bs\rho_c^t)}e^{\pm\ri k_{\ell z} (z-z_c^t)}=4\pi\sum\limits_{n'=0}^{\infty}\sum\limits_{|m'|=0}^{n'} \ri^{n'}(\pm 1)^{\mu}j_{n'}(k_{\ell}r_t)Y_{n'}^{m'}(\theta_t,\varphi_t)\widehat{P}_{n'}^{m'}\Big(\frac{k_{\ell z}}{k_{\ell}}\Big)e^{-\ri m'\alpha},
\end{equation*}
where $\mu=(n'+m')(\bmod\; 2)$.
Substituting into \eqref{melayerupgoingimage}, the multipole expansion is translated to local expansion \eqref{lelayerimage} via
\begin{equation}\label{metoleimage1}
L_{nm}^{\uparrow*}=\sum\limits_{n'=0}^{\infty}\sum\limits_{|m'|=0}^{n'}T_{nm,n'm'}^{\uparrow*}M_{n'm'}^{\uparrow*},\; L_{nm}^{\downarrow*}=(-1)^{n+m}\sum\limits_{n'=0}^{\infty}\sum\limits_{|m'|=0}^{n'}T_{nm,n'm'}^{\downarrow*}M_{n'm'}^{\downarrow*},
\end{equation}
and the multipole-to-local translation operators are given in integral forms as follows
\begin{equation}\label{metoleimage2}
\begin{split}
T_{nm,n'm'}^{\uparrow*}=&\frac{(-1)^{n'}}{2\pi }\int_0^{\infty}\int_0^{2\pi}k_{\rho}\frac{\widetilde{\mathcal E}_{\ell\ell'}^{\uparrow}(\bs r_c^t, \bs r_c^{\uparrow*})\sigma_{\ell\ell'}^{\uparrow*}(k_{\rho})}{k_{\ell z}}Q_{nm}^{n'm'}(k_{\rho})e^{\ri (m'-m)\alpha}d\alpha dk_{\rho},\\
T_{nm,n'm'}^{\downarrow*}=&\frac{(-1)^{n'}}{2\pi }\int_0^{\infty}\int_0^{2\pi}k_{\rho}\frac{\widetilde{\mathcal E}_{\ell\ell'}^{\downarrow}(\bs r_c^t, \bs r_c^{\downarrow*})\sigma_{\ell\ell'}^{\downarrow*}(k_{\rho})}{k_{\ell z}}Q_{nm}^{n'm'}(k_{\rho})e^{\ri (m'-m)\alpha}d\alpha dk_{\rho},
\end{split}
\end{equation}
where
$$Q_{nm}^{n'm'}(k_{\rho})=\ri^{n+n'+1}\widehat{P}_{n}^{m}\Big(\frac{k_{\ell z}}{k_{\ell}}\Big)\widehat{P}_{n'}^{m'}\Big(\frac{k_{\ell'z}}{k_{\ell'}}\Big).$$
{We note that the convergence of the Sommerfeld-type integrals in \eqref{metoleimage2} is ensured by the conditions in \eqref{imagecentercond}. }

\begin{rem}
{\bf Interpretation of Polarization Sources:} In special situations such as the half space (a two layer medium) with an impedance boundary condition on the interface, the reaction
fields can be expressed in terms of those from point and line image charges located on the opposite side of the interface from the targets \cite{cho2018heterogeneous} \cite{cai2013computational}.
However, in multilayered cases considered here, the reaction fields, given in terms of complicated integral expressions, in general can not be
expressed in terms of explicit image charges. However, we can still introduce the polarization sources { with locations given in (\ref{eqpolarizedsource}) and Fig. \ref{sourceimages}}, which will represent
the effect of the reaction fields and are given locations based on the convergence behaviors of the MEs and LEs in \eqref{melayerapp} for the corresponding reaction field components.	For the practical FMM implementation purpose, the locations of the polarization sources so defined enable us to decide whether the corresponding MEs can be used for the far field of the reaction components based on the distance between the far field location and the polarization sources, and thus, make the extension of FMMs to source interactions in layered media straightforward.
\end{rem}

\subsection{A FMM algorithm for sources in layered media}
Let $\mathscr{P}_{\ell}=\{(Q_{\ell j},\boldsymbol{r}_{\ell j}),$ $j=1,2,\cdots
,N_{\ell}\}$ be a group of source particles distributed in the $\ell$-th layer
of a multi-layered medium with $L+1$ layers (see Fig. \ref{layerstructure}).
The interactions between all $N:=N_{0}+N_{1}+\cdots+N_{L}$ particles are given by
the summations
\begin{equation}
\Phi_{\ell}(\boldsymbol{r}_{\ell i})=\Phi_{\ell}^{\text{free}}(\boldsymbol{r}_{\ell
	i})+\sum\limits_{\ell^{\prime}=0}^{L-1}[\Phi_{\ell\ell^{\prime}}^{\uparrow\uparrow
}(\boldsymbol{r}_{\ell i})+\Phi_{\ell\ell^{\prime}}^{\downarrow\uparrow
}(\boldsymbol{r}_{\ell i})]+\sum\limits_{\ell^{\prime}=1}^{L}[\Phi_{\ell\ell^{\prime}}^{\uparrow\downarrow
}(\boldsymbol{r}_{\ell i})+\Phi_{\ell\ell^{\prime}}^{\downarrow\downarrow
}(\boldsymbol{r}_{\ell i})],\quad \label{totalinteraction}%
\end{equation}
for $i=1,2,\cdots,N_{\ell}$, $\ell=0,1,\cdots,L$, where
\begin{equation}
\begin{split}
&  \Phi_{\ell}^{\text{free}}(\boldsymbol{r}_{\ell i}):=\sum\limits_{j=1,j\neq
	i}^{N_{\ell}}Q_{\ell j}h_{0}^{(1)}(k_{\ell}|\boldsymbol{r}_{\ell
	i}-\boldsymbol{r}_{\ell j}|),\quad  \Phi_{\ell\ell^{\prime}}^{**}(\boldsymbol{r}_{\ell i}):=\sum
\limits_{j=1}^{N_{\ell^{\prime}}}Q_{\ell^{\prime}j}u_{\ell\ell^{\prime}%
}^{**}(\boldsymbol{r}_{\ell i},\boldsymbol{r}_{\ell^{\prime}j}%
).
\end{split}
\end{equation}
Here, $u_{\ell\ell^{\prime}}^{**}(\bs r, \bs r^{\prime})$ are
the reaction components of the layered Green's function in the $\ell
$-th layer due to a point source in the
$\ell^{\prime}$-th layer. We omit the factor $\frac{\ri k_{\ell}}{4\pi}$
in $u_{\ell\ell^{\prime}}^{**}(\bs r, \bs r^{\prime})$ to maintain consistency with the spherical Hankel function used for the free space component.
General formulas for $u_{\ell\ell^{\prime}}^{**}(\bs r, \bs r^{\prime})$ are given in \eqref{greenfuncomponent}-\eqref{totaldensity}.
\begin{figure}[ht!]
	\includegraphics[scale=0.6]{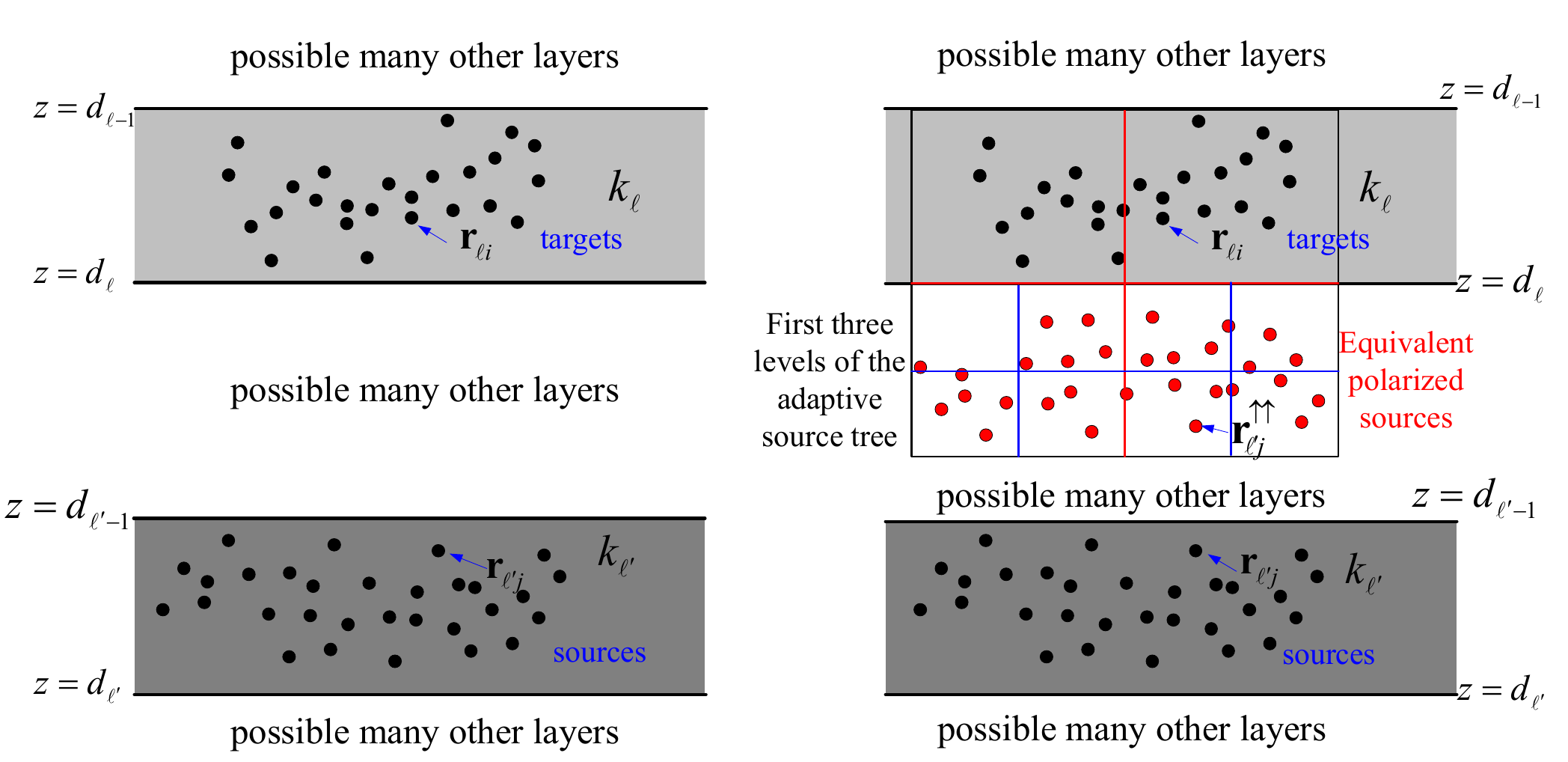}
	\caption{Equivalent polarization sources $\bs r_{\ell'j}^{\uparrow\uparrow}$ and boxes in source tree.}
	\label{polarizedsource}
\end{figure}

Since the reaction components of Green's function in multi-layer media have different
expressions \eqref{greenfuncomponent} for source and target particles in
different layers, it is necessary to perform calculation individually for
interactions between any two groups of particles among the $L+1$ groups
$\{\mathscr{P}_{\ell}\}_{\ell=0}^{L}$. Here, we only consider the contribution from the reaction components and that from free space components will be calculated using free space FMM. Without loss of generality, let us
focus on the computation of the $\uparrow\uparrow$ component in
$\ell$-th layer due to sources in $\ell^{\prime}$-th layer, i.e., $\Phi_{\ell\ell^{\prime}}^{\uparrow\uparrow}(\boldsymbol{r}_{\ell i})$ for all $i=1, 2, \cdots, N_{\ell}$ where we have source particles $\mathscr P_{\ell^{\prime}}$ and target particles $\mathscr P_{\ell}$. According to the discussion in the last section, we use the equivalent polarization sources (see. Fig. \ref{polarizedsource})
$$\boldsymbol{r}_{\ell^{\prime}j}^{\uparrow\uparrow}:=(x_{\ell' j}, y_{\ell' j}, d_{\ell}-(z_{\ell' j}-d_{\ell'})),$$
and re-express $\Phi_{\ell\ell'}^{\uparrow\uparrow}(\bs r_{\ell i})$ as
\begin{equation}\label{layerfield}
\Phi_{\ell\ell'}^{\uparrow\uparrow}(\bs r_{\ell i})=\sum
\limits_{j=1}^{N_{\ell^{\prime}}}Q_{\ell^{\prime}j}\tilde u_{\ell\ell^{\prime}%
}^{\uparrow\uparrow}(\boldsymbol{r}_{\ell i},\boldsymbol{r}_{\ell^{\prime}j}^{\uparrow\uparrow}),\quad i=1,2,\cdots,N_{\ell}.
\end{equation}
A very important feature of the equivalent polarization sources $\{\boldsymbol{r}_{\ell^{\prime}j}^{\uparrow\uparrow}\}_{j=1}^{N_{\ell'}}$ is that they are all separated from the corresponding targets $\{\bs r_{\ell i}\}_{i=1}^{N_{\ell}}$ by the interface $z=d_{\ell}$. As a matter of fact,  they are located on different side of the interface $z=d_{\ell}$, see. Fig. \ref{polarizedsource} (right). Equivalent polarization sources \eqref{eqpolarizedsource} defined for other reaction components also have this property.

The framework of the free space FMM with ME \eqref{melayerupgoingimage}, LE \eqref{lelayerimage}, M2L translation \eqref{metoleimage1}-\eqref{metoleimage2} and free space ME and LE center shifting allow us to use FMM for the computation of reaction component $\Phi_{\ell\ell'}^{\uparrow\uparrow}(\bs r_{\ell i})$. FMM for other reaction components can be treated in the same manner. In the FMM for reaction components, for each reaction component, a corresponding large box is defined to include all equivalent polarization sources and target particles where the adaptive tree structure will be built by a bisection procedure, see. Fig. \ref{polarizedsource} (right). { Note that the validity of the ME \eqref{melayerupgoingimage}, LE \eqref{lelayerimage} and M2L translation \eqref{metoleimage1} used in the algorithm imposes restrictions \eqref{imagecentercond} on the centers, accordingly. This can be ensured by setting the largest box for the specific reaction component to be equally divided by the interface between equivalent polarization sources and targets, see. Fig. \ref{polarizedsource} (right). Thus, the largest box for the FMM implementation will be different for different reaction component.  With this setting, all source and target boxes of level higher than zeroth level in the adaptive tree structure will have centers below or above the interface $z=d_{\ell}$, accordingly. }
The fast multipole algorithm for the computation of the reaction component $\Phi_{\ell\ell'}^{\uparrow\uparrow}(\bs r_{\ell i})$ is summarized in Algorithm 1.
\begin{algorithm}\label{algorithm1}
	\caption{FMM for general component $\Phi_{\ell\ell'}^{\uparrow\uparrow}(\bs r_{\ell i}), i=1, 2, \cdots, N_{\ell}$}
	\begin{algorithmic}
		\State Determine $z$-coordinates of equivalent polarization sources for all source particles.
		\State Generate an adaptive hierarchical tree structure with polarization sources $\{Q_{\ell'j}, \bs r_{\ell'j}^{\uparrow\uparrow}\}_{j=1}^{N_{\ell'}}$, targets $\{\bs r_{\ell i}\}_{i=1}^{N_{\ell}}$ and pre-compute tables of $S^{\uparrow\uparrow}_{nm,\nu\mu}$ in \eqref{M2Ltable}.
		\State{\bf Upward pass:}
		\For{$l=H \to 0$}
		\For{all boxes $j$ on source tree level $l$ }
		\If{$j$ is a leaf node}
		\State{form the free-space ME using Eq. \eqref{melayerupgoingimage}.}
		\Else
		\State form the free-space ME by merging children's expansions using the free-space center shift translation operator \eqref{metome}.
		\EndIf
		\EndFor
		\EndFor
		\State{\bf Downward pass:}
		\For{$l=1 \to H$}
		\For{all boxes $j$ on target tree level $l$ }
		\State shift the LE of $j$'s parent to $j$ itself using the free-space shifting \eqref{metome}.
		\State collect interaction list contribution using the source box to target box translation operator in Eq. \eqref{metoleimage1} and \eqref{metoletableform} with pre-computed tables of $S^{\uparrow\uparrow}_{nm,\nu\mu}$.
		\EndFor
		\EndFor
		\State {\bf Evaluate Local Expansions:}
		\For{each leaf node (childless box)}
		\State evaluate the local expansion at each particle location.
		\EndFor
		\State {\bf Local Direct Interactions:}
		\For{$i=1 \to N$ }
		\State compute Eq. \eqref{layerfield} of target particle $i$ in the neighboring boxes using pre-computed tables of $S^{\uparrow\uparrow}_{nm,\nu\mu}$.
		\EndFor
	\end{algorithmic}
\end{algorithm}
All the interaction given by \eqref{totalinteraction} will be obtained by calculating all components and summing them up.

\subsection{Implementation details for efficiency}
In this section, we will discuss some implementation details regarding the computation of double integrals involved in the multipole and local expansions and multipole-to-local translations.
Especially, they can be simplified by using the following identity
\begin{equation}
J_n(z)=\frac{1}{2\pi \ri^n}\int_0^{2\pi}e^{\ri z\cos\theta+\ri n\theta}d\theta.
\end{equation}
In particular, multipole expansion functions in \eqref{mebasis} can be simplified as
\begin{equation*}
\begin{split}
\widetilde F_{nm}^{\uparrow*}(\bs r, \bs r_c^{\uparrow*})=\frac{(-\ri)^n\ri^{m+1}e^{\ri m\varphi_s^{\uparrow*}}}{4\pi}\int_0^{\infty}k_{\rho}J_{m}(k_{\rho}\rho_s^{\uparrow*})\frac{\widetilde{\mathcal{Z}}_{\ell\ell'}^{\uparrow}(z, z_c^{\uparrow*})\sigma_{\ell\ell'}^{\uparrow*}(k_{\rho})}{k_{\ell z}}\widehat{P}_n^m\Big(\frac{k_{\ell'z}}{k_{\ell'}}\Big) dk_{\rho},\\
\widetilde F_{nm}^{\downarrow*}(\bs r, \bs r_c^{\downarrow*})=\frac{(-\ri)^n\ri^{m+1}e^{\ri m\varphi_s^{\downarrow*}}}{4\pi}\int_0^{\infty}k_{\rho}J_{m}(k_{\rho}\rho_s^{\downarrow*})\frac{\widetilde{\mathcal{Z}}_{\ell\ell'}^{\downarrow}(z, z_c^{\downarrow*})\sigma_{\ell\ell'}^{\downarrow*}(k_{\rho})}{k_{\ell z}}\widehat{P}_n^m\Big(\frac{k_{\ell'z}}{k_{\ell'}}\Big) dk_{\rho},
\end{split}
\end{equation*}
and the expression \eqref{lecoeffimage} for local expansion coefficients can be simplified as
\begin{equation*}
\begin{split}
L_{nm}^{\uparrow*}= (-1)^m\ri^{n-m+1}e^{-\ri m(\pi+\varphi_t^{\uparrow*})}\int_0^{\infty}k_{\rho}J_{m}(k_{\rho}\rho_t^{\uparrow*})\frac{\widetilde{\mathcal Z}_{\ell\ell'}^{\uparrow}(z_c^t, z'_{\uparrow*})\sigma_{\ell\ell'}^{\uparrow*}(k_{\rho})}{k_{\ell z}}\widehat{P}_n^m\Big(\frac{k_{\ell z}}{k_{\ell}}\Big) dk_{\rho},\\
L_{nm}^{\downarrow*}= (-1)^{n}\ri^{n-m+1}e^{-\ri m(\pi+\varphi_t^{\downarrow*})}\int_0^{\infty}k_{\rho}J_{m}(k_{\rho}\rho_t^{\downarrow*})\frac{\widetilde{\mathcal Z}_{\ell\ell'}^{\downarrow}(z_c^t, z'_{\downarrow*})\sigma_{\ell\ell'}^{\downarrow*}(k_{\rho})}{k_{\ell z}}\widehat{P}_n^m\Big(\frac{k_{\ell z}}{k_{\ell}}\Big) dk_{\rho},
\end{split}
\end{equation*}
where $({\rho}_s^{**}, \varphi_s^{**})$ and $({\rho}_t^{**}, \varphi_t^{**})$ are polar coordinates of $\bs r'_{**}-\bs r_c^{**}$ and $\bs r'_{**}-\bs r_c^t$ projected on $xy$ plane. Moreover, the multipole to local translation \eqref{metoleimage2} can be simplified as
\begin{equation}\label{me2lesimplified}
\begin{split}
T_{nm,n'm'}^{\uparrow*}&=C_{n'mm'}^{\uparrow*}\int_0^{\infty}k_{\rho}J_{m'-m}(k_{\rho}\rho_{ts}^{\uparrow*})\frac{\widetilde{\mathcal{Z}}_{\ell\ell'}^{\uparrow}(z_c^t, z_c^{\uparrow*})\sigma_{\ell\ell'}^{\uparrow*}(k_{\rho})}{k_{\ell z}}Q_{nm}^{n'm'}(k_{\rho}) dk_{\rho},\\
T_{nm,n'm'}^{\downarrow*}&=C_{n'mm'}^{\downarrow*}\int_0^{\infty}k_{\rho}J_{m'-m}(k_{\rho}\rho_{ts}^{\downarrow*})\frac{\widetilde{\mathcal{Z}}_{\ell\ell'}^{\downarrow}(z_c^t, z_c^{\downarrow*})\sigma_{\ell\ell'}^{\downarrow*}(k_{\rho})}{k_{\ell z}}Q_{nm}^{n'm'}(k_{\rho}) dk_{\rho},
\end{split}
\end{equation}
where $C_{n'mm'}^{**}=(-1)^{n'} \ri^{m'-m}e^{\ri(m'-m)\varphi_{ts}^{**}}$ and $(\rho_{ts}^{**}, {\varphi}_{ts}^{**})$ is the polar coordinates of $\bs r_c^t-\bs r_c^{**}$ projected on $xy$ plane.
Denoting
\begin{equation}\label{uniformintegral}
\begin{split}
I_{nm,n'm'}^{\uparrow*}(\rho, z, z')=\int_0^{\infty}k_{\rho}J_{m'-m}(k_{\rho}\rho)\frac{\widetilde{\mathcal{Z}}_{\ell\ell'}^{*}(z, z')\sigma_{\ell\ell'}^{\uparrow*}(k_{\rho})}{k_{\ell z}}Q_{nm}^{n'm'}(k_{\rho}) dk_{\rho},\\
I_{nm,n'm'}^{\downarrow*}(\rho, z, z')=\int_0^{\infty}k_{\rho}J_{m'-m}(k_{\rho}\rho)\frac{\widetilde{\mathcal{Z}}_{\ell\ell'}^{*}(z, z')\sigma_{\ell\ell'}^{\downarrow*}(k_{\rho})}{k_{\ell z}}Q_{nm}^{n'm'}(k_{\rho}) dk_{\rho},
\end{split}
\end{equation}
then we have
\begin{equation}
\begin{split}
&\widetilde F_{nm}^{**}(\bs r, \bs r_c^{**})=\frac{(-1)^n\ri^{m}e^{\ri m\varphi_s^{**}}}{(\sqrt{4\pi})^3}I_{nm,00}^{**}(\rho_s^{**}, z, z_c^{**}), \\
&L_{nm}^{\uparrow*}=\frac{ \ri^{m}e^{-\ri m(\pi+\varphi_t^{\uparrow*})}}{\sqrt{4\pi}}I_{nm,00}^{\uparrow*}(\rho_t^{\uparrow*}, z_c^t, z'_{\uparrow*}),\\
&L_{nm}^{\downarrow*}=\frac{ (-1)^{n-m}\ri^{m}e^{-\ri m(\pi+\varphi_t^{\downarrow*})}}{\sqrt{4\pi}}I_{nm,00}^{\downarrow*}(\rho_t^{\downarrow*}, z_c^t, z'_{\downarrow*}),\\
& T_{nm,n'm'}^{**}=C_{n'mm'}^{**}I_{nm,n'm'}^{*\uparrow}(\rho_{ts}^{**}, z_c^t, z_c^{**}).
\end{split}
\end{equation}

The FMM demands efficient computation of Sommerfeld-type integrals $I_{nm,n'm'}^{**}$ defined in \eqref{uniformintegral}. It
is clear that they have oscillatory integrands with pole singularities in $\sigma_{\ell\ell'}^{**}(k_{\rho})$ due to
the existence of surface waves. For a long time, much effort has been
made on the computation of this type of integrals, including using ideas from
high-frequency asymptotics, rational approximation, contour deformation (cf.
\cite{cai2013computational,cai2000fast,cho2012parallel,okhmatovski2004evaluation,paulus2000accurate}%
), complex images (cf.\cite{fang1988discrete,paulus2000accurate,ochmann2004complex,alparslan2010closed}%
), and methods based on special functions (cf. \cite{koh2006exact}) or
physical images (cf.
\cite{li1996near,ling2000discrete,o2014efficient,lai2016new}).
These integrals are convergent when the target and
source particles are not exactly on the interfaces of a layered medium.
Contour deformation with high order quadrature rules could be used for direct
numerical computation at runtime. However, this becomes prohibitively expensive due to a
large number of integrals needed in the FMM. In fact, $O(p^{4})$ integrals
will be required for each source box to target box translation. Moreover, the
involved integrand decays more slowly as the order of the involved associated Legendre function increases. The length of contour needs to be very long to obtain a
required accuracy.

Note that $I_{nm,n'm'}^{**}(\rho, z, z')$ is a smooth function with respect to $(\rho,z,z^{\prime})$. It is feasible to make a pre-computed
table on a fine grid and then use interpolation to obtain approximations for
the integrals. If we make pre-computed tables for all $|m|\leq n, |m'|\leq n'$, $n, n'=0, 1, \cdots, p$, there will be $(p+1)^4$  3-D tables to be stored. However,  the number of pre-computed tables can be reduced to $4(p+1)(2p+1)$.

Let
\begin{equation}
c_{nm}=\sqrt{\frac{2n+1}{4\pi}\frac{(n-m)!}{(n+m)!}},\quad  a_{nm}^j=\frac{(-1)^{n-j}(2j)!c_{nm}}{2^nj!(n-j)!(2j-n-m)!}.
\end{equation}
then
\begin{equation}\label{identitypolyexp}
\frac{c_{00}(x^2-1)^n}{2^nn!}=\sum\limits_{k=0}^na_{00}^jx^{2j},\quad
\frac{c_{nm}}{2^nn!}\frac{d^{n+m}}{dx^{n+m}}(x^2-1)^n=\sum\limits_{j=\lceil\frac{n+m}{2}\rceil}^na_{nm}^jx^{2j-n-m}.
\end{equation}
Recall the Rodrigues' formula
\begin{equation}
\widehat P_n^m(x)=\frac{c_{nm}}{2^nn!}(1-x^2)^{\frac{m}{2}}\frac{d^{n+m}}{dx^{n+m}}(x^2-1)^n,
\end{equation}
for $0\leq m\leq n$, and use the fact $k_{\ell z}=\sqrt{k_{\ell}^2-k_{\rho}^2}$ and \eqref{identitypolyexp}, we obtain
\begin{equation}
\widehat P_n^m\Big(\frac{k_{\ell z}}{k_{\ell}}\Big)=\begin{cases}
\displaystyle\sum\limits_{s=0}^{n-r}b_{nm}^s\Big(\frac{k_{\rho}}{k_{\ell}}\Big)^{m+2s}, & n+m=2r,\\
\displaystyle\frac{k_{\ell}}{k_{\ell z}}\sum\limits_{s=0}^{n-r}b_{nm}^s\Big(\frac{k_{\rho}}{k_{\ell}}\Big)^{m+2s}, & n+m=2r+1,
\end{cases}
\end{equation}
where
\begin{equation}
b_{nm}^s=\sum\limits_{j=q}^{n}\frac{(-1)^{s}a_{nm}^j(j-r)!}{s!(j-r-s)!},\quad q=\max\Big(\Big\lceil\frac{n+m}{2}\Big\rceil, s+r\Big).
\end{equation}
Furthermore
\begin{equation}
\widehat P_n^m\Big(\frac{k_{\ell z}}{k_{\ell}}\Big)\widehat P_{n^{\prime}}^{m^{\prime}}\Big(\frac{k_{\ell^{\prime} z}}{k_{\ell^{\prime}}}\Big)=\sum\limits_{s=0}^{n-r+n'-r'}A_{nn'mm'}^sk_{\rho}^{m+m'+2s}\Big(\frac{k_{\ell}}{k_{\ell z}}\Big)^{\mu}\Big(\frac{k_{\ell'}}{k_{\ell' z}}\Big)^{\nu},
\end{equation}
where $\mu=(n+m)(\bmod\; 2)$, $\nu=(n'+m')(\bmod\; 2)$,
\begin{equation}\label{asslegendreexpcoef}
A_{nn'mm'}^{s}=\sum\limits_{t=\max(s-n'+r', 0)}^{\min(s, n-r)}\frac{b_{nm}^tb_{n'm'}^{s-t}}{k_{\ell}^{m+2t}k_{\ell'}^{m'+2(s-t)}}.
\end{equation}
For $-n\leq m<0$ or $-n'\leq m'<0$, similar formulas can be obtained by using the fact  $\widehat{P}_n^m(z)=(-1)^m\widehat{P}_n^{-m}(z)$ for $m<0$. In summary, we have
\begin{equation}
\widehat P_n^m\Big(\frac{k_{\ell z}}{k_{\ell}}\Big)\widehat P_{n^{\prime}}^{m^{\prime}}\Big(\frac{k_{\ell^{\prime} z}}{k_{\ell^{\prime}}}\Big)=\sum\limits_{s=0}^{n-r+n'-r'}A_{nn'mm'}^sk_{\rho}^{|m|+|m'|+2s}\Big(\frac{k_{\ell}}{k_{\ell z}}\Big)^{\mu}\Big(\frac{k_{\ell'}}{k_{\ell' z}}\Big)^{\nu},
\end{equation}
for all $n, n'=0, 1, \cdots, $ and $-n\leq m\leq n$, $-n'\leq m'\leq n'$ where $\mu=(n+|m|)(\bmod\; 2)$, $\nu=(n'+|m'|)(\bmod\; 2)$, $r=\big\lfloor (n+|m|)/2\big\rfloor, \quad r'=\big\lfloor(n'+|m'|)/2\big\rfloor,$
and
$$
A_{nn'mm'}^{s}=\sum\limits_{t=\max(s-n'+r', 0)}^{\min(s, n-r)}\frac{\tau_m\tau_{m'}b_{n|m|}^tb_{n'|m'|}^{s-t}}{k_{\ell}^{|m|+2t}k_{\ell'}^{|m'|+2(s-t)}}, \quad \tau_{\nu}=\begin{cases}
1, & \nu\geq 0,\\
(-1)^{-\nu}, & \nu<0.
\end{cases}
$$
Here, we also use the notation $A_{nn'mm'}^s$ since it is equal to the coefficients defined in \eqref{asslegendreexpcoef} for $m, m'\geq 0$.
Define integrals
\begin{equation}\label{M2Ltable}
\begin{split}
\mathcal S_{nm,ij}^{\uparrow*}(\rho, z, z')=\int_0^{\infty}k_{\rho}^{m+1}J_{n}(k_{\rho}\rho)\frac{\widetilde{\mathcal{Z}}_{\ell\ell'}^{\uparrow}(z, z')\sigma_{\ell\ell'}^{\uparrow*}(k_{\rho})}{k_{\ell z}}\Big(\frac{k_{\ell}}{k_{\ell z}}\Big)^i\Big(\frac{k_{\ell'}}{k_{\ell' z}}\Big)^j dk_{\rho},\\
\mathcal S_{nm,ij}^{\downarrow*}(\rho, z, z')=\int_0^{\infty}k_{\rho}^{m+1}J_{n}(k_{\rho}\rho)\frac{\widetilde{\mathcal{Z}}_{\ell\ell'}^{\downarrow}(z, z')\sigma_{\ell\ell'}^{\downarrow*}(k_{\rho})}{k_{\ell z}}\Big(\frac{k_{\ell}}{k_{\ell z}}\Big)^i\Big(\frac{k_{\ell'}}{k_{\ell' z}}\Big)^j dk_{\rho}.
\end{split}
\end{equation}
Then
\begin{equation}\label{metoletableform}
I_{nm,n'm'}^{**}(\rho, z, z')=\ri^{n+n'+1}\sum\limits_{s=0}^{n-r+n'-r'}A_{nn'mm'}^s\mathcal S_{m-m',|m|+|m'|+2s,\mu\nu}^{**}(\rho, z, z'),
\end{equation}
where $\mu=(n+|m|)(\bmod\; 2)$, $\nu=(n'+|m'|)(\bmod\; 2)$.
We pre-compute integrals $\mathcal S_{nm,\mu\nu}^{**}(\rho, z, z')$ on a 3-D grid $\{\rho_{i},z_{j},z_{k}^{\prime}\}$ in the
domain of interest for all $n=0,1,\cdots,p$; $m=0, 1, \cdots, 2n+1$; $\mu, \nu=0,1$. Then, a polynomial interpolation is performed for the computation of integrals in the FMM.

To compute Sommerfeld-type integrals $\mathcal S_{nm,\mu\nu}^{**}(\rho, z, z')$, it is
typical to deform the integration contour by pushing it away from the real
line into the fourth quadrant of the complex $k_{\rho}$-plane to avoid branch
points and poles in the integrand, see Fig. \ref{deformedcontour}. Here, we use a piecewise smooth contour
which consists of two segments:
\begin{equation}
\Gamma_{1}:\;\;\{k_{\rho}=\ri t,-b\leq t\leq0\},\quad\Gamma_{2}%
:\;\;\{k_{\rho}=t-\ri b,\quad0\leq t<\infty\}.
\end{equation}
We truncate $\Gamma_{2}$ at a point $t_{max}>0$, where the integrand has
decayed to a user specified tolerance. As an example, we plot the density $\sigma_{11}^{**}(k_{\rho})$ along $\Gamma_{2}$ (see, Fig. \ref{deformedcontour} (c)). The three layers case with
$k_{0}=1.5,k_{1}=0.8,k_{2}=2.0, d_0=0, d_1=-2.0$ and density given
in Section 3 is used. We can see that the density $\sigma_{11}^{**}(k_{\rho})$ is bounded.
\begin{figure}[ptbh]
	\center
	\subfigure[Contour in $k_{\rho}$ plane]{\includegraphics[scale=0.35]{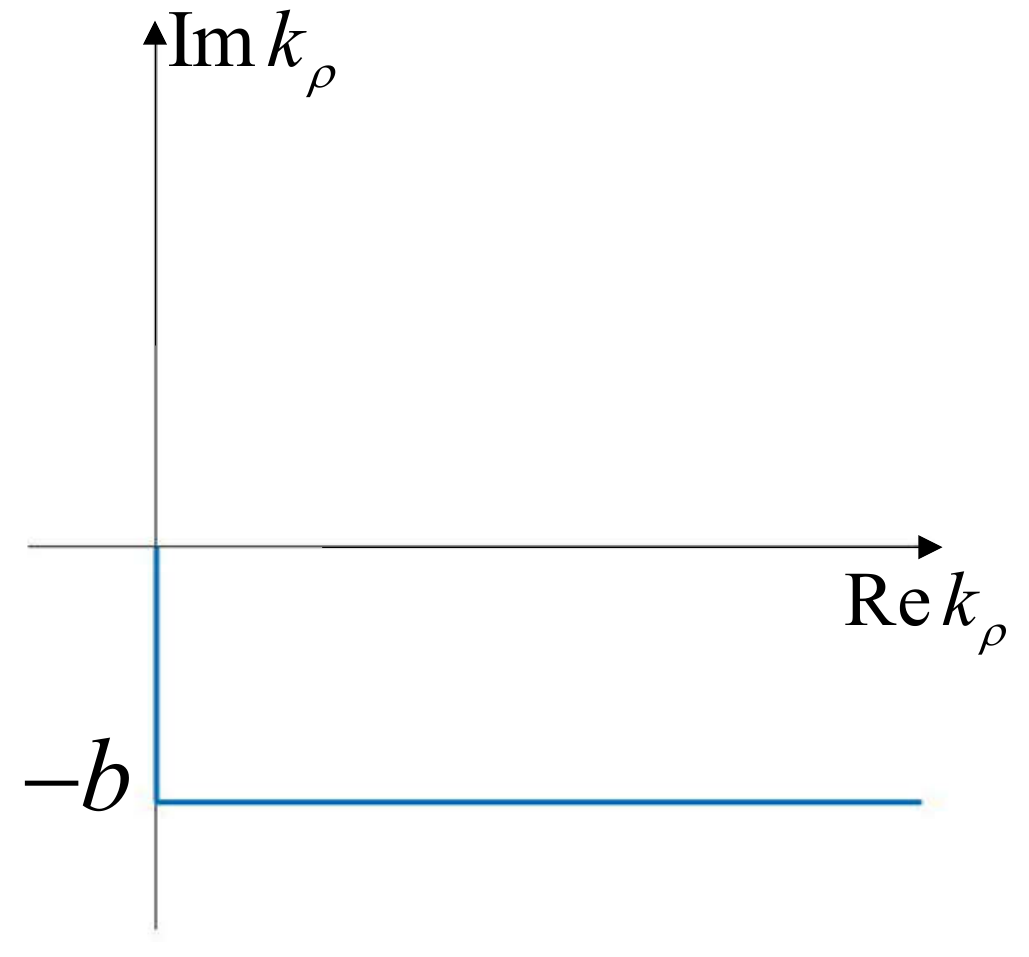}}\quad
	\subfigure[Contour in $k_{z}$ plane]{\includegraphics[scale=0.35]{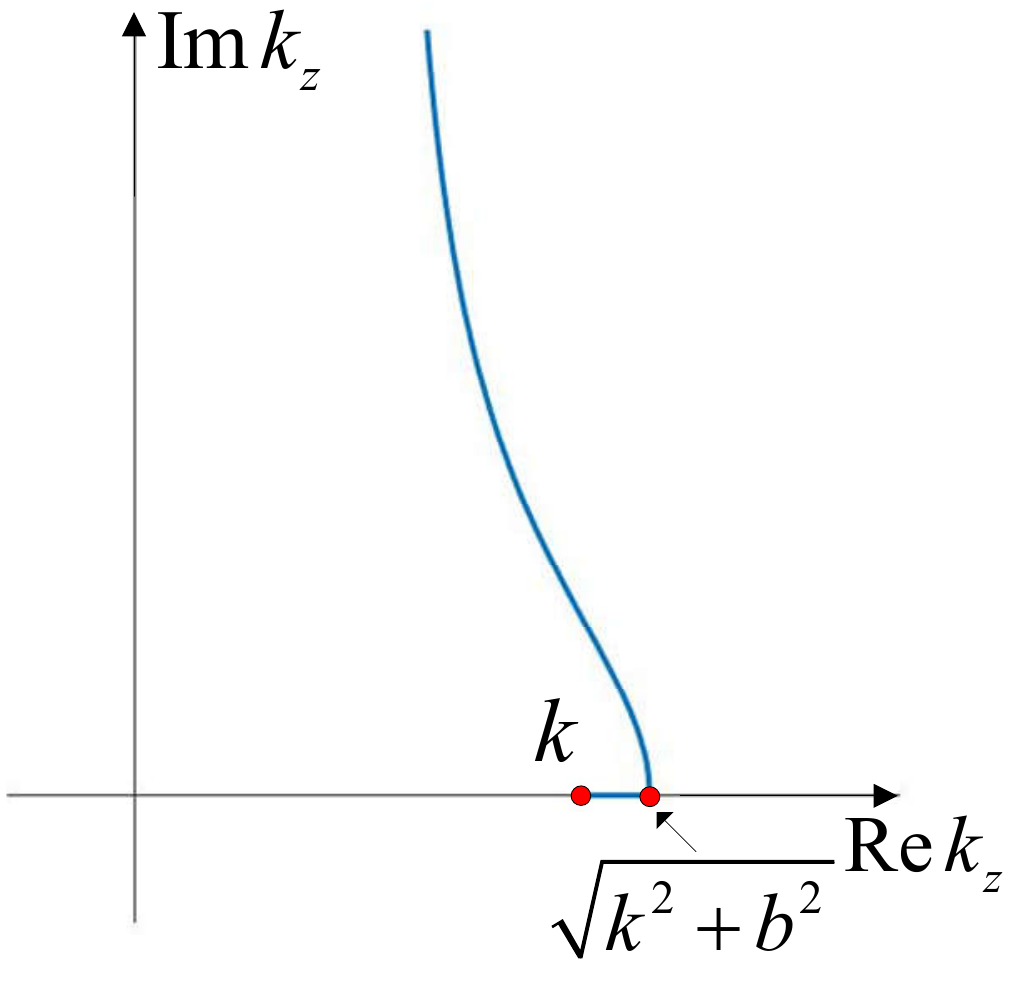}}\quad
	\subfigure[real part of $\sigma_{11}^{**}(k_{\rho})$ along $\Gamma_2$]{\includegraphics[scale=0.35]{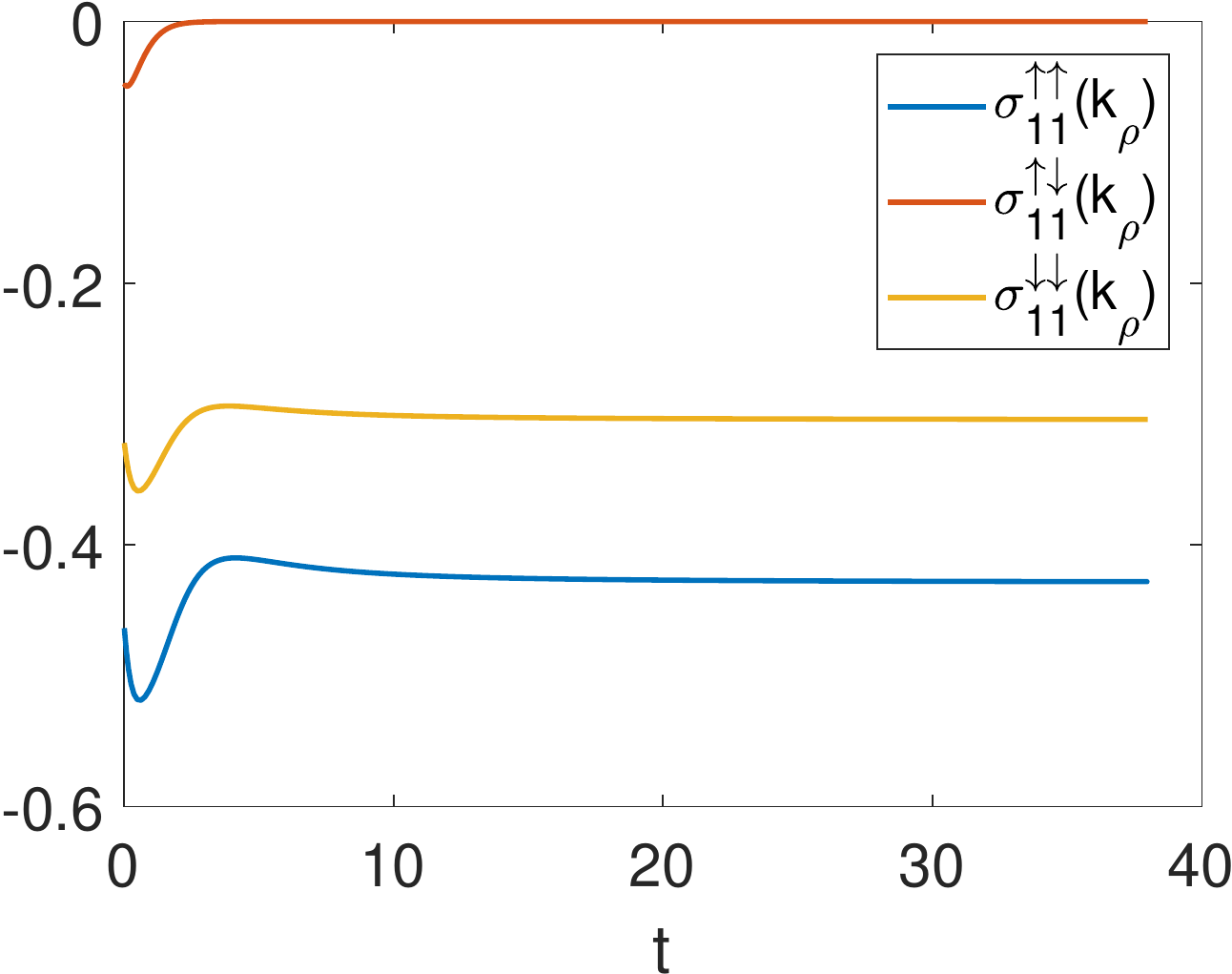}}
	\caption{Deformed contour for the computation of Sommerfeld-type integrals.}%
	\label{deformedcontour}%
\end{figure}

\section{Numerical results}

In this section, we present numerical results to demonstrate the performance
of the proposed FMM for time harmonic wave scattering in layered media.
This algorithm is implemented based on an open-source adaptive FMM package DASHMM
\cite{debuhr2016dashmm} on a
workstation with two Xeon E5-2699 v4 2.2 GHz processors (each has 22 cores)
and 500GB RAM using the gcc compiler version 6.3.

We test the problem in three layers media with interfaces placed at $z_{0}=0$, $z_{1}=-1.2$. Particles
are set to be uniformly distributed in irregular domains which are obtained by shifting the domain determined by $r=0.5-a+\frac{a}{8}(35\cos^{4}\theta-30\cos^{2}\theta+3)$
with $a=0.1,0.15,0.05$ to new centers $(0,0,0.6)$, $(0,0,-0.6)$
and $(0,0,-1.8)$, respectively (see
Fig. \ref{fmmperformance} (a) for the cross section of the domains). All particles are generated by keeping the
uniform distributed particles in a larger cube within corresponding irregular
domains. In the layered media, the wave numbers are $k_0=1.2$, $k_1=1.5$, $k_2=1.8$. Let $\widetilde{\Phi}_{\ell}%
(\boldsymbol{r}_{\ell i})$ be the approximated values of $\Phi_{\ell
}(\boldsymbol{r}_{\ell i})$ calculated by FMM.
For accuracy test, we put
$N=912+640+1296$ particles in the irregular domains in three layers see Fig. \ref{fmmperformance} (a).
Convergence rates for the relative $\ell^2$ error ($Err_2$) and relative maximum error ($Err_{max}$) against $p$ are depicted in Fig. \ref{fmmperformance} (b).
The CPU time for the computation of all three free space components $\{\Phi^{free}_{\ell}(\boldsymbol{r}_{\ell i})\}_{\ell=0}^2$, three selected reaction components
$\{\Phi^{\uparrow\uparrow}_{00},\Phi^{\uparrow\uparrow}_{11},\Phi^{\downarrow\downarrow}_{22}\}$ and all sixteen reaction components $\Phi^{**}_{\ell\ell'}(\boldsymbol{r}_{\ell i})$ with truncation $p=4$ are compared in Fig. \ref{fmmperformance} (c) for up to 3 millions particles. It shows that all of them have an $O(N)$ complexity while the CPU time for the computation of reaction components has a much smaller linear scaling constant {due to the fact that most of the equivalent polarization sources are well-separated with the targets.} CPU time with multiple cores is given in Table \ref{Table:ex1three} and it shows that, due to the small amount of CPU time in computing the reaction components, the speedup of the parallel computing is mainly decided by the computation of the free space components. Here, we only use parallel implementation within the computation of each component, and all reaction components are computed independently. Therefore, it is straightforward to implement a version of the code which computes all components in parallel.

{ {\bf {Pre-computed translation operator tables:}} Given the truncation order $p$, there will be $4(p+1)(2p+1)$ 3-D tables to be pre-computed, whose size
depends on the required accuracy. In our numerical tests, we use tables with $30\times 30\times 30$ mesh with an interpolation polynomial of order $3$.
Then, each table have $91^3$ integrals to be calculated. For each integral,  the infinite interval is truncated at $t_{max}=100$ and a mesh with 51 sub-intervals
and 30 Gaussian quadrature points in each sub-interval is used.  It takes about 8.5 seconds on our workstation to compute one 3-D table
as all the integrals can be calculated in parallel.
}

\begin{figure}[ptbh]
	\center
	\subfigure[distribution of particles]{\includegraphics[scale=0.29]{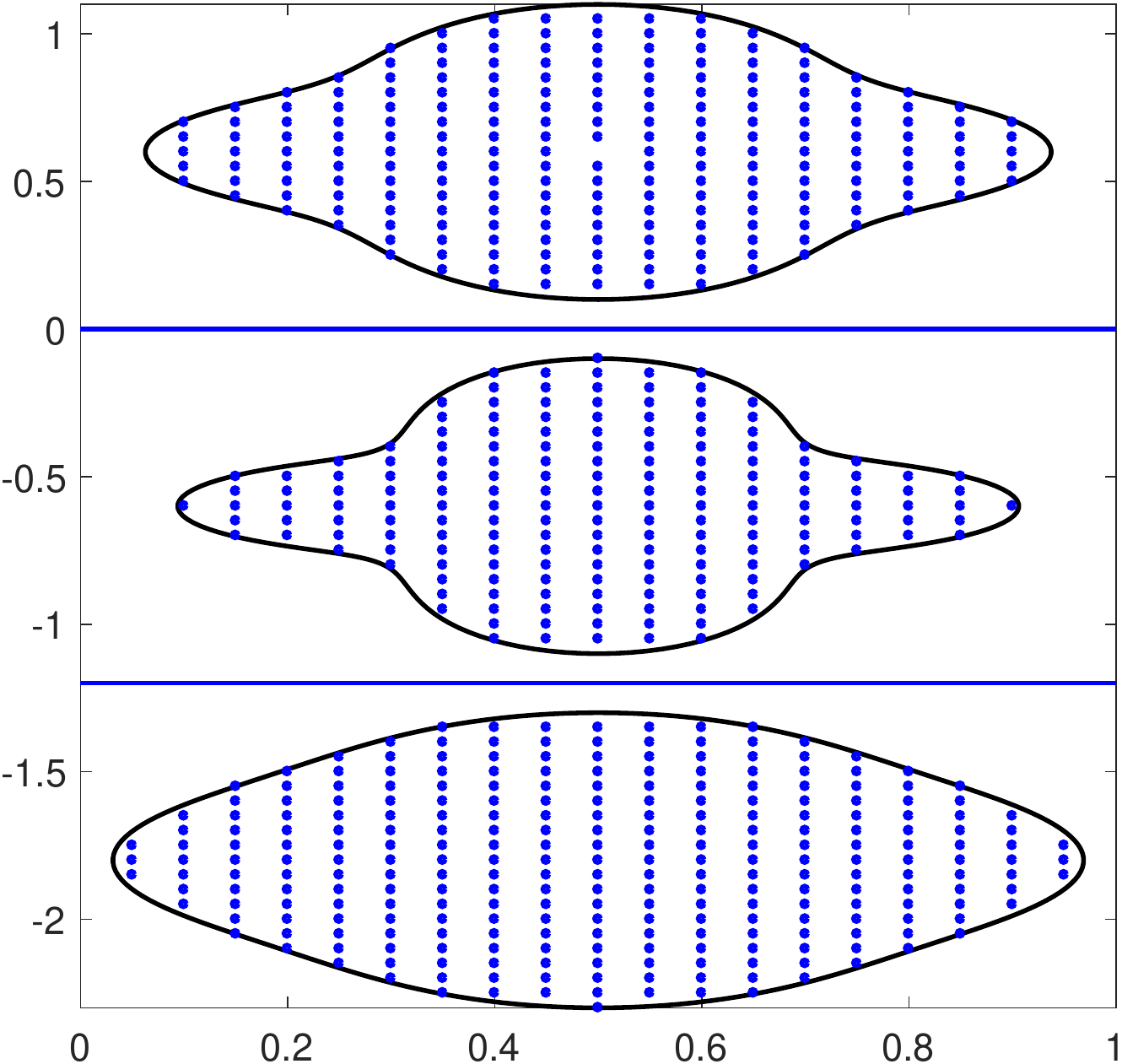}}\quad
	\subfigure[convergence rates vs. $p$]{\includegraphics[scale=0.2]{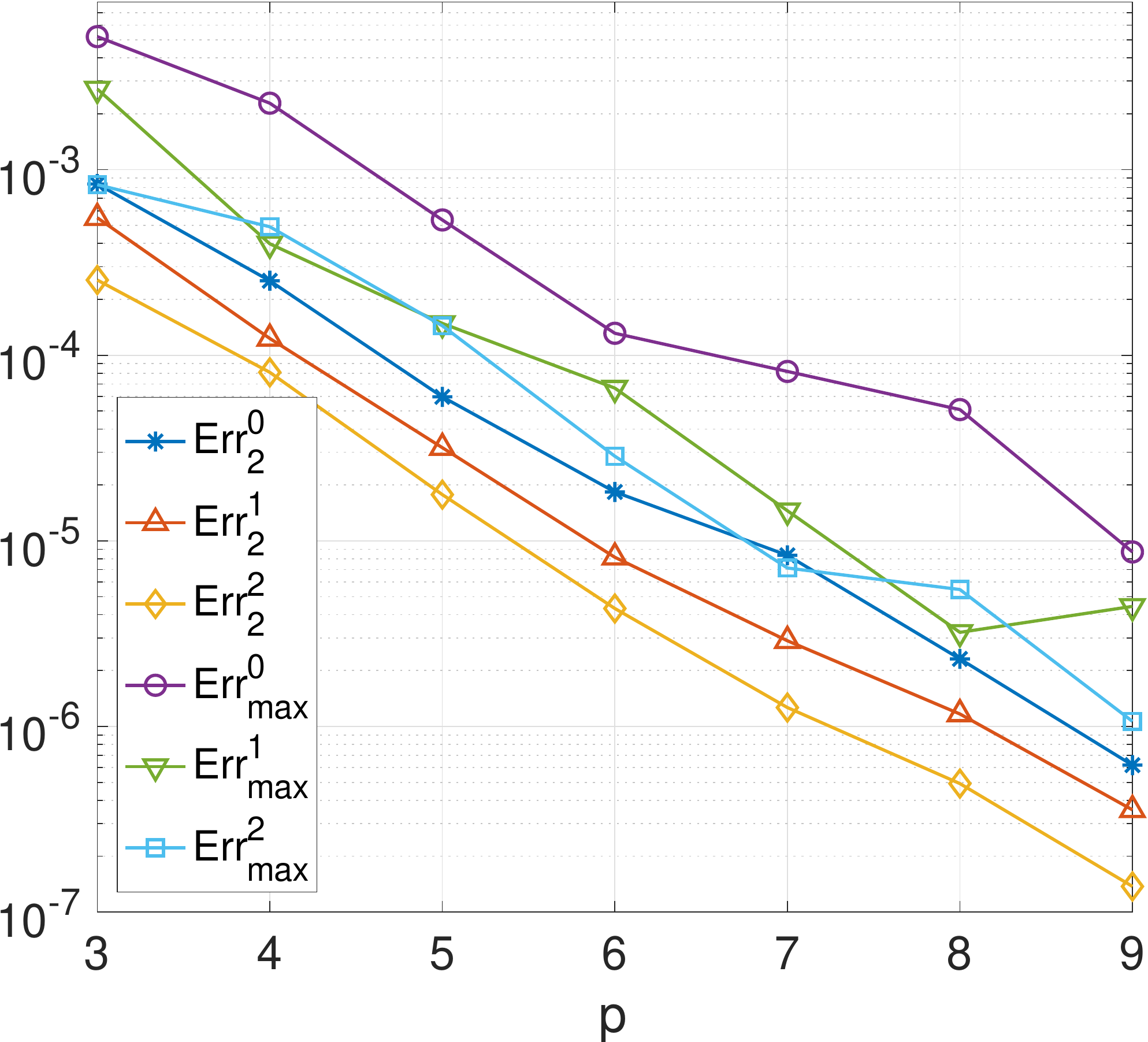}}
	\subfigure[CPU time vs. $N$]{\includegraphics[scale=0.2]{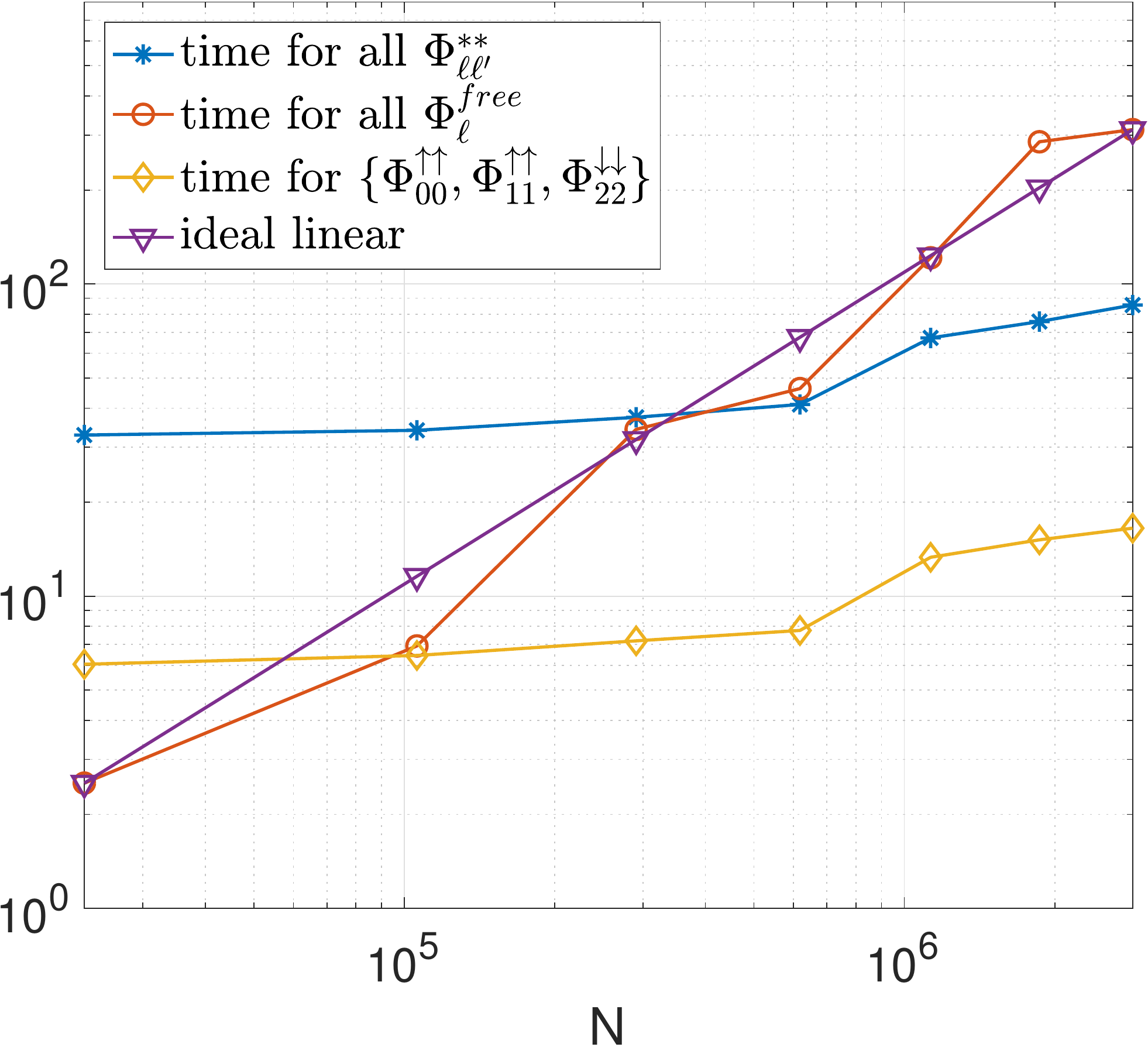}}
	\caption{Performance of FMM for a three layers media problem.}%
	\label{fmmperformance}%
\end{figure}
\begin{table}[ptbhptbhptbhptbh]
	\centering {\small
		\begin{tabular}
			[c]{|c|c|c|c|}\hline
			cores & $N$ & time for all $\{\Phi_{\ell}^{free}\}_{\ell=0}^2$ & time for $\{\Phi^{\uparrow\uparrow}_{00},\Phi^{\uparrow\uparrow}_{11},\Phi^{\downarrow\downarrow}_{22}\}$ 			\\\hline
			\multirow{4}{*}{1} &618256 & 46.01 & 4.60  \\\cline{2-4}
			& 1128556 & 120.9 &12.13  \\\cline{2-4}
			& 1862568 & 282.4 & 13.88 \\\cline{2-4}
			& 2861288 & 306.9 & 15.22 \\\hline
			\multirow{4}{*}{6} & 618256 & 8.75 & 2.77 \\\cline{2-4}
			& 1128556 & 21.95 &4.00  \\\cline{2-4}
			& 1862568 & 52.19 & 4.57  \\\cline{2-4}
			& 2861288 & 57.04 & 5.15  \\\hline
			\multirow{4}{*}{36} & 618256 & 2.29 & 2.74  \\\cline{2-4}
			& 1128556 & 5.00 &3.21  \\\cline{2-4}
			& 1862568 & 11.45 & 3.59  \\\cline{2-4}
			& 2861288 & 12.60 & 3.90  \\\hline
		\end{tabular}
	}
	\caption{Comparison of CPU time with multiple cores ($p=4$).}%
	\label{Table:ex1three}%
\end{table}
\section{Conclusion and discussion}

In this paper, we presented {an $O(N)$} fast multipole method for  the  calculation of the discretized integral operator for
the low-frequency Helmholtz equation in 3-D layered media. The layered media Green's function is decomposed into a free space and four types
of reaction field component(s). Using the spectral form of the layered Green's functions and the Funk-Hecke identity, we developed
new multipole expansion of $O(p^{2})$ terms for the far field of the reaction components,  which can be associated with
polarization sources at specific locations for the reaction field components. Multipole to local translation
operators are also developed for the reaction fields. As a result, the traditional ME-based FMM can be applied to both
the free space part and the reaction fields once the polarization sources are combined with the original sources.
Due to the separation of the polarization and the original source charges by a material interface, the
computational cost from the reaction field parts is only a fraction of that of the FMM for the free space part.
For a given layered structure, besides the one time pre-computations of interpolation tables for the translation operator,
computing the wave interactions of many sources in layered media costs the
same as that for the wave interactions in the free space.

{
The complexity of the FMM on $L$ for a $L-$layer medium will scale as $O(L^2)$ as the reaction
component expression is different for different combination of source and target
layers, the source-target interactions will need to be computed, separately.
However, the FMMs for different group of reaction components can be done in parallel,
and in practice $L$ is usually on the order of 10.  Also, to reduce the pre-computation time of the tables for the M2L translation matrices, which involve the Sommerfeld integrals \eqref{M2Ltable} with smooth integrands along the deformed contour, faster algorithms can
be developed by using better quadrature rule (e.g., Clenshaw-Curtis-Filon-type method)
and the recurrence formula of Bessel functions (see details in \cite{FMMlaplace}).

The FMM developed here only considered low-frequency Helmholtz equations. Similar to the free space case,
the multipole-expansion based FMM in this paper can not handle high frequency wave source interactions well as the number of terms
in the MEs will increase in proportion to $k*D$ \cite{darve2000} where $k$ is the wave number and $D$ is the size of the
scatterer.  To address this difficulty, plane waves expansions could be introduced to yield diagonal forms
for the translation operators in the FMM to produce an $O(N\log N)$ fast algorithm \cite{rokhlin1993}\cite{chew1997}.
Though, the plane wave expansion may suffer a low-frequency breakdown when interaction between sources within sub-wavelength
distance is computed for large objects with small fine structures. Therefore, it is important to develop a
 stable FMM applicable to a broad range of frequencies from low to high wave numbers without the low-frequency break down while still
applicable to high frequency scattering. Various methods have been proposed toward this goal again for the free space case,  including hybrid approaches
combining multipole and plane wave expansions \cite{cheng2004}\cite{chew2018} as well as FMMs using stable plane wave expansions \cite{darve2004}
and inhomogeneous plane waves \cite{huchew2000}.

 As a future work, as the Green's functions for the layered media are given in terms of plane waves
via Sommerfeld integration, we will study broadband FMMs for the layered media. } For an immediate task, we shall carry out error analysis for the new MEs and M2L operators for the reaction components for the Helmholtz equations in 3-D
layered media, extending our results for the 2-D Helmholtz equations \cite{zhang2018exponential}.

\section*{Acknowledgments}
The authors thank Min Hyung Cho for helpful discussions.

\end{document}